\documentclass[a4paper,11pt,reqno]{amsart}
\usepackage[a4paper,hscale=0.7,vscale=0.75,centering]{geometry}
\usepackage{amssymb}

\newcommand{\erre}{\mathbb{R}}
\renewcommand{\r}{\mathbb{R}}

\newcommand{\E}{\mathbb{E}}
\newcommand{\EE}{\mathcal{E}}
\renewcommand{\L}{\mathcal{L}}
\renewcommand{\P}{\mathbb{P}}
\newcommand{\ip}[2]{\langle #1,#2 \rangle}
\newcommand{\bip}[2]{\left\langle #1,#2 \right\rangle}

\newcommand{\osc}{\mathop{\mathrm{osc}}\nolimits}

\newcommand{\var}{\mathop{\mathrm{Var}}\nolimits}

\newtheorem{prop}{Proposition}[section]
\newtheorem{thm}[prop]{Theorem}
\newtheorem{coroll}[prop]{Corollary}
\newtheorem{lemma}[prop]{Lemma}

\theoremstyle{definition}

\newtheorem{rmk}[prop]{Remark}

\numberwithin{equation}{section}

\begin{document}

\title[Approximations of evolving probability measures]{Quantitative
  approximations of evolving probability measures and sequential
  Markov Chain Monte Carlo methods}

\author{Andreas Eberle}
\address[Andreas Eberle]{Institut f\"ur Angewandte Mathematik,
  Universit\"at Bonn, Endenicher Allee 60, D-53115 Bonn, Germany}
\email{eberle@uni-bonn.de}
\urladdr{http://www.uni-bonn.de/$\sim$eberle}

\author{Carlo Marinelli}
\address[Carlo Marinelli]{Institut f\"ur Angewandte Mathematik,
  Universit\"at Bonn, Endenicher Allee 60, D-53115 Bonn, Germany, and
  Facolt\`a di Economia, Universit\`a di Bolzano, Piazza Universit\`a
  1, I-39100 Bolzano, Italy}
\urladdr{http://www.uni-bonn.de/$\sim$cm788}

\date{July 13, 2011}

\begin{abstract}
  We study approximations of evolving probability measures by an
  interacting particle system. The particle system dynamics is a
  combination of independent Markov chain moves and importance
  sampling/resampling steps. Under global regularity conditions, we
  derive non-asymptotic error bounds for the particle system
  approximation. In a few simple examples, including high dimensional
  product measures, bounds with explicit constants of feasible size
  are obtained. Our main motivation are applications to sequential
  MCMC methods for Monte Carlo integral estimation.
\end{abstract}



\thanks{We would like to thank the anonymous referees for detailed
  reports and very helpful comments on the first version of this
  article. This work was partially supported by the
  Sonderforschungsbereich 611, Bonn. The second-named author also
  gratefully acknowledges the support of the DAAD}

\maketitle

\section{Introduction}
\subsection{Evolving probability measures}
Let $(\mu_t)_{t \in [0,\infty)}$ denote a family of mutually
absolutely continuous probability measures on a set $S$. To keep the
presentation as simple and non-technical as possible, we assume that
$S$ is finite. Motivated by Monte Carlo methods for sequential
estimation of expectation values with respect to the probability
measures $\mu_t$ (see e.g. \cite{cappe,DM,DMDJ,DdFG} and references
therein), we will recall how to obtain Fokker-Planck type evolution
equations on the space of probability measures on $S$ that are
satisfied by $\mu_t$, and how to approximate these equations by
interacting particle systems. The main purpose of this paper is to
bound the error of the particle system approximations by an $L^p$
approach (see Theorems \ref{thm:main}, \ref{thm:main2} and
\ref{thm:main3} below).

Sequential Monte Carlo (SMC) methods that combine Markov Chain Monte
Carlo (MCMC) and Importance Sampling/Resampling methods to approximate
a given sequence $(\mu_t)$ of probability measures are used in a
variety of applications, see for instance
\cite{chop-stat,DMDJ,casella} and references therein. There is by now
a substantial literature on approximation properties of corresponding
particle system discretizations, cf. \cite{cappe,DM,MM} and the
references cited below. Nevertheless, our mathematical understanding
of SMC methods is still far more superficial than that of traditional
MCMC methods, where, at least for some specific models, sharp bounds
for mixing times, approximation errors and dependence on the dimension
have been derived. The $L^p$ approach to controlling the approximation
error that we propose here is a first step towards more quantitative
results that might be useful in particular in studying dimensional
dependence. In contrast to most of the literature on SMC methods (see
however \cite{MM,Rou-phd,Rou}), we focus on the continuous time case.

\smallskip

We assume that the measures are represented in the form
\begin{equation}
\mu_t(x) = \frac 1{Z_t}\,\exp \left(-\mathcal{U}_t(x)\right)\,\mu_0 (x),
\qquad t\geq 0,
\end{equation}
where $Z_t$ is a normalization constant, and $(t,x) \mapsto
\mathcal{U}_t(x)$ is a given function on $[0,\infty )\times S$ that is
continuously differentiable in the first variable.  If, for example,
$\mathcal{U}_t(x)=t\,\mathcal{U}(x)$ for some function
$\mathcal{U}:S\to \mathbb{R}$, then $(\mu_t)_{t\geq 0}$ is the
exponential family corresponding to $\mathcal{U}$ and $\mu_0$.  Let
\[
H_t(x) := -\frac{\partial}{\partial t}\log\mu_t(x)
= -\frac{\partial}{\partial t}
\log\frac{\mu_t(x)}{\mu_0(x)}
\]
denote the negative logarithmic time derivative of the measures
$\mu_t$. Note that
\begin{equation} \label{eq:timedep} 
  \mu_t(x) = \exp
  \left(-\int_0^tH_s(x)\,ds\right)\;\mu_0 (x)\,,
\end{equation}
and
\begin{equation}     \label{eq:centered}
  \ip{H_t}{\mu_t} = -\frac{d}{dt}\mu_t(S) = 0
  \qquad \text{for all } t\geq 0,
\end{equation}
where
\[
\ip{f}{\nu} := \int_S f\,d\nu\ = \sum_{x\in S} f(x)\,\nu (x)
\]
denotes the integral of a function $f:S\to\r $ w.r.t. a measure
$\nu $ on $S$. In particular,
\[
H_t = \frac{\partial}{\partial t}\mathcal{U}_t - \ip{\frac{\partial}{\partial
t}\mathcal{U}_t}{\mu_t}.
\]
In the applications we have in mind, the functions $\mathcal{U}_t$ are
given explicitly. Hence $H_t$ is known explicitly up to an additive
time-dependent constant. The evaluation of this constant, however,
would require computing an integral w.r.t. $\mu_t$.

\smallskip

If all the functions $H_t$, $t\geq 0$, vanish then $\mu_t=\mu_0$
for all $t\geq 0$. In this case the measures are invariant for a
Markov transition semigroup $(p_t)_{t\geq 0}$, i.e.,
\[
\mu_s p_{t-s} = \mu_t \qquad \text{for any } t\geq s\geq 0,
\]
provided the generator $\mathcal L$ of $(p_t)_{t\geq 0}$ satisfies
$\mu_0 \mathcal{L}=0$, i.e.
\[
\sum_{x \in S} \mu_0(x) \mathcal{L}(x,y)=0
\qquad \text{for any } y \in S.
\]
This fact is exploited in Markov Chain Monte Carlo methods for
approximating expectation values w.r.t. the measure $\mu_0$.  The
particle systems studied below can be applied for the same purpose
when the measures $\mu_t$ are time-dependent.

\subsection{Fokker-Planck equation and particle system
approximation}
To obtain approximations of the measures $\mu_t$, we consider
generators ($Q$-matrices) ${\mathcal L}_t$, $t\ge 0$, of a
time-inhomogeneous Markov process on $S$ satisfying the detailed
balance conditions
\begin{equation}
  \label{eq:db}
  \mu_t(x){\mathcal L}_t(x,y) = \mu_t(y){\mathcal L}_t(y,x)
      \quad \forall\ t\geq0, \; x,y\in S.
\end{equation}
For example, ${\mathcal L}_t$ could be the generator of a Metropolis
dynamics w.r.t.~$\mu_t$, i.e.,
\[
{\mathcal L}_t(x,y)\ =\ K_t(x,y)\cdot\min\left(
\frac{\mu_t(y)}{\mu_t (x)} ,1\right)\mbox{ for }x\neq y,
\]
${\mathcal L }_t(x,x)=-\sum_{y\neq x}{\mathcal L}_t(x,y)$, where the
proposal matrix $K_t$ is a given symmetric transition matrix on $S$.
In the sequel we will use the notation $\mathcal{L}_t^*\mu$ to denote
the adjoint action of the generator on a probability measure $\mu$, i.e.,
\[
(\mathcal{L}_t^*\mu)(y) := (\mu\mathcal{L}_t)(y) = 
\sum_{x \in S} \mu(x) \mathcal{L}_t(x,y).
\]
By (\ref{eq:db}), $\mathcal{L}_t^\ast\mu_t=0$, i.e.,
\[
\ip{{\mathcal L_t}f}{\mu_t} = 0
\qquad \text{for any } f:S\to\r \text{ and } t\geq 0.
\]
We fix non-negative constants $\lambda_t$, $t\geq 0$, such that $t
\mapsto \lambda_t$ is continuous. Since the state space $S$ is finite,
the measures $\mu_t$ are the \emph{unique} solution of the evolution equation
for measures
\begin{equation}
  \label{eq:FP}
  \frac{\partial}{\partial t}\nu_t =
  \lambda_t\,{\mathcal L}_t^\ast \nu_t - H_t\nu_t
\end{equation}
with initial condition $\nu_0=\mu_0$. In general, solutions of
(\ref{eq:FP}) are not necessarily probability measures, even if
$\nu_0$ is a probability measure. Therefore, we consider the equation
\begin{equation}
  \label{eq:FPnorm}
  \frac{\partial}{\partial t}\eta_t\ =\
  \lambda_t\, {\mathcal L}_t^\ast \eta_t\, -\, H_t\eta_t\, +\,
  \ip{H_t}{\eta_t}\,\eta_t
\end{equation}
satisfied by the normalized measures $\eta_t=\frac{\nu_t}{\nu_t(S)}$.
Note that, by (\ref{eq:centered}), $\mu_t$ also solves
(\ref{eq:FPnorm}). Moreover, if $\eta_t$ is a solution of
(\ref{eq:FPnorm}), then
\[
\nu_t = \exp\left( -\int_0^t \ip{H_s}{\eta_s}\,ds \right)\,\eta_t
\]
is the unique solution of (\ref{eq:FP}) with initial condition
$\nu_0=\eta_0$.

\medskip

The Fokker-Planck equation (\ref{eq:FPnorm}) is an evolution equation
for probability measures which, in contrast to the unnormalized
equation, is not modified by adding constants to the functions $H_t$.
We now introduce interacting particle systems that discretize the
evolution equations (\ref{eq:FPnorm}) and (\ref{eq:FP}). Consider
right continuous time-inhomogeneous Markov processes $(X_t^N,
\mathbb{P})$, $N\in\mathbb{N}$, with state space $S^N$ and generators
at time $t$ given by
\begin{equation}
  \label{eq:generator} 
  \begin{split}
    \mathcal{L}_t^N\varphi (x_1,\ldots ,x_N) &=
    \lambda_t\sum_{i=1}^N\mathcal{L}_t^{(i)}\varphi (x_1,\ldots,x_N)\\
    &\quad +\frac 1N\sum_{i,j=1}^N(H_t(x_i)-H_t(x_j))^+ (\varphi
    (x^{i\to j})-\varphi (x)).
  \end{split}
\end{equation}
Here $x=(x_1,\ldots ,x_N)\in S^N$ and
\[
(x^{i\to j})_k\ =\ \left\{\begin{array}{ll} x_k &\mbox{if }k\neq i,\\
x_j &\mbox{if }k=i. \end{array}   \right.
\]
Moreover, $\mathcal{L}_t^{(i)}$ stands for the operator
$\mathcal{L}_t$ applied to the $i$-th component of $x$. Thus the
components $X_{t,i}^{N}$, $i=1,\ldots , N$, of the process $X_t^{N}$
move like independent Markov processes with generator
$\lambda_t\mathcal{L}_t$ and are occasionally replaced by components
with a lower value of $H_t$. Note that to compute the generator (and
hence to simulate the Markov process) it is enough to know the
functions $H_t$ up to an additive constant.

\smallskip

Discretizations of interacting particle systems of a similar type are
widely used in applications, where mostly the time parameter is
discrete. Variants appear in the literature under different names,
including sequential Monte Carlo methods (e.g. in
\cite{DMDJ,DouMou,DdFG}), population Monte Carlo algorithms
\cite{CGMR,DGMR,Rou-phd,Rou}, Feynman-Kac particle models
\cite{cerone,DM,MM}), particle filters \cite{BBL,BLB,chop-stat}),
etc. Theoretical properties of these Monte Carlo methods and, in
particular, the asymptotics as $N \to \infty$, have been studied
intensively (mostly in discrete time), see e.g. \cite{cappe,DM} for an
overview, and \cite{cerone,JasDou} for more recent results. The
continuous time case has been investigated in \cite{MM,Rou-phd,Rou}.

The Markov processes $(X_t^N,\P)$ introduced above are continuous-time
analogues of a particular type of sequential Monte Carlo samplers
which have been introduced and studied systematically in \cite{DMDJ}
(cf. also \cite{chop-stat,DMDJ2,Jarz,Neal}). One major motivation for
the use of SMC samplers is the estimation of expectation values with
respect to multimodal distributions where traditional MCMC methods
fail due to metastability problems. The processes $(X_t^N,\P)$ have
the additional property that the underlying generator at time $t$
satisfies detailed balance w.r.t $\mu_t$. In this case, the resulting
sequential MCMC methods are also related to several
multi-level sampling methods, including parallel tempering
\cite{Geyer91,HukNem,MZ} and the equi-energy sampler \cite{KZW}. The
detailed balance condition is not necessarily required for
applications, but it fixes a clear framework that is the foundation
for our $L^p$ approach developed below.

\smallskip

It is essentially well-known (see \cite{MM}) that if the initial
distributions of the Markov processes $(X_t^N,\P)$ are the $N$-fold
products $\pi^N$ of a probability measure $\pi$ on $S$, then almost
surely, the empirical distributions
\begin{equation}
  \label{eq:EMP}
  \eta_t^N = \frac 1N\sum_{i=1}^N\delta_{X^N_{t,i}}
\end{equation}
and the reweighted empirical distributions
\begin{equation}
  \label{eq:REWEMP}
  \nu_t^N = \exp\left(-\int_0^t\ip{H_s}{\eta_s^N } \right)\,\eta_t^N
\end{equation}
converge to the solutions of the equations (\ref{eq:FPnorm}) and
(\ref{eq:FP}) with initial conditions $\eta_0=\nu_0=\pi$, see also
Corollary \ref{cor:estf} below. As a consequence, simulating the
Markov process $X^N_t$ with initial distribution $\mu_0^N$ yields a
Monte Carlo method for approximating sequentially the probability
measures $\mu_t$, $t\ge 0$, which can be viewed as a combination of
Markov Chain Monte Carlo and Importance Sampling/Resampling.

\subsection{Quantitative convergence bounds}
Our main aim is to quantify more explicitly the approximation
properties of the particle systems with initial distribution
$\mu_0^N$. There is a substantial literature on asymptotic properties
of corresponding particle system approximations, see
e.g. \cite{DM,MM,Rou-phd} and references therein. In particular, a law
of large numbers type convergence theorem and a corresponding central
limit theorem have been established in \cite{DMG,MM} for a related
particle system approximation, cf. also \cite{Rou}. A crucial question
for algorithmic applications, however, are quantitative bounds on the
approximation error
\begin{equation}
  \label{eq:ERR}
  \langle f,\eta_t^N\rangle -\langle f,\mu_t\rangle
\end{equation}
for a given function $f:S\to\mathbb{R}$ and fixed $N$ that incorporate
some more explicit control of the constants. For example, the
dependence of the bounds on the dimension in product models is very
relevant.

The central limit theorem in \cite{MM} yields bounds for the
approximation error (\ref{eq:ERR}) asymptotically as $N\to\infty$ (at
least for a modified particle system). In \cite{Rou} corresponding
non-asymptotic estimates are given but without quantifying the
constants. We also refer to \cite{cerone} for some more recent
non-asymptotic estimates under strong mixing conditions in discrete
time.
In this respect, several important questions still remain open:
\begin{itemize}
\item The expression for the asymptotic variance in the central limit
  theorem derived in \cite{MM} is not very explicit, as it involves
  $L^2$ norms of an associated Feynman-Kac semigroup.  Methods that
  allow to bound this expression \emph{efficiently} in a general setup
  and in concrete models have to be developed.
\item For applications it is crucial to derive more explicit
  \emph{non-asymptotic bounds} (i.e. bounds for fixed $N$), because
  the asymptotic estimates could be misleading when only a limited
  number of particles is available. To the best of our knowledge such
  bounds have been proven so far only under partially restrictive
  minorization (see \cite{white}) or strong mixing conditions
  involving constants that are not very explicit, highly
  dimension-dependent, and far from optimal. In general, tracking the
  constants in the proof of the CLT in \cite{MM} shows that these
  could be of order up to $\exp\int_0^t \osc(H_s)\,ds$, where
  $\osc(H_s):=\sup H_s - \inf H_s$ stands for the oscillation of
  $H_s$. In nearly all interesting applications this quantity is
  extremely large. Hence although the existing results give useful
  indications on scope and limits of SMC methods, the rigorous
  verification of a given error bound for a realistic number $N$ of
  particles/replicas is still an open problem in many simple concrete
  models.
\item Dimensional dependence on product spaces is an important issue,
  cf. \cite{BBL,BCJ,BLB}. Rigorous results about the dependence on the
  dimension of error bounds for SMC methods are still missing, and
  might be out of reach for the existing techniques.
\end{itemize}
It is well-known from the theory of reversible Markov processes that a
convergence analysis based only on total variation estimates and
Dobrushin contraction coefficients is possible but it has several
drawbacks. In particular, substantial contractivity w.r.t. the total
variation norm often takes place only after a certain number of steps
(cutoff phenomena, cf. e.g. \cite{Dia-cutoff,DLP}). This limits the
applicability if one is interested in arguments based on single or
even infinitesimal time steps. Moreover, minorization conditions that
are often imposed in this context are crude and typically dimension
dependent. Therefore, in this article we develop the foundations of an
alternative approach to establish non-asymptotic bounds for the
particle system approximations, which enables us to prove bounds with
a reasonable dependence on the dimension for product models, see
Example 2 below. The approach we propose is based on a consequent
application of $L^p$ estimates instead of uniform estimates for
Feynman-Kac propagators. In \cite{EM-proc} (cf. also \cite{St-var}),
an $L^2$ approach has been considered to quantify asymptotic stability
properties of the Fokker-Planck equation. When studying the error of
particle system approximations, we are forced to leave the $L^2$
framework and to work with various $L^p$ norms. A key tool are the
$L^p$ estimates for Feynman-Kac propagators that have been derived in
\cite{JMAA}.

\subsection{Outline} 
The main results of our work are stated in Section \ref{sec:mr}. Here
we also consider examples where the approximation errors can be
quantified explicitly. Section \ref{sec:mp} contains the derivation of
an explicit formula for the variances of the estimators
$\ip{f}{\nu_t^N}$, see Proposition \ref{thm:bondo} below. This is
based on martingale arguments developed in \cite{MM}. In Section
\ref{sec:p1} we apply the formula to prove Theorem \ref{thm:main}
below, which is a non-asymptotic bound for the variances. Finally, in
Section \ref{sec:p2} we combine this bound with the results from
\cite{JMAA} to prove the bounds in Theorems \ref{thm:main2} and
\ref{thm:main3} below.

\section{Main results} \label{sec:mr} To state our results in detail
let us consider the Markov process $(X_t^N,\mathbb{P})$ with initial
distribution $\mu_0^N$. To derive error bounds for the particle system
approximation it is convenient to consider at first the error for the
Monte Carlo estimates based on the reweighted empirical distributions
$\nu_t^N$ defined in (\ref{eq:REWEMP}). Following closely the
reasoning in \cite{MM}, we first note that, by a martingale argument,
it can be shown that $\ip{f}{\nu_t^N}$ is an unbiased estimator of
$\ip{f}{\mu_t}$ for any function $f:S \to \erre$ and $t\geq 0$, and an
explicit formula for the variance can be given.

\subsection{An expression for the variance}
To state the formula for the variance, we introduce Feynman-Kac type
transition operators $q_{s,t}$ related to the dynamics. For $0 \leq s
\leq t < \infty$ and a function $f:S\to\erre$, let $q_{s,t}f(x)$
denote the unique solution of the backward equation
\begin{equation}
  \label{eq:BWE}
 - \frac{\partial}{\partial s} q_{s,t}f = \lambda_s \L_s q_{s,t}f - H_sq_{s,t}f,
  \qquad s \in [0,t],
\end{equation}
with terminal condition $q_{t,t}f=f$. It can be shown that
$q_{s,t}f$ is also the unique solution of the corresponding
forward equation
\begin{equation}     \label{eq:FWE}
  \frac{\partial}{\partial t} q_{s,t}f = q_{s,t}(\lambda_t\L_t f - H_t f),
  \qquad t\in [s, \infty ),
\end{equation}
with initial condition $q_{s,s}f=f$. As a consequence, a
probabilistic representation of $q_{s,t}$ is given by the Feynman-Kac
formula
\begin{equation}       \label{eq:FK}
  (q_{s,t}f)(x) = \E_{s,x} \big[ e^{-\int_s^t H_r(X_r)\,dr} f(X_t)
  \big]\qquad\mbox{for all }x\in S,
\end{equation}
where $(X_t)_{t \geq s}$ is a time-inhomogeneous Markov process
w.r.t. $\P_{s,x}$ with generator $\L_t$ and initial condition
$X_s=x$ $\P_{s,x}$-a.s., see e.g. \cite{GS-II}, \cite{Guli}.
The next proposition is an adaptation of results in \cite[\S3.3]{MM}
to our slightly modified setting.
\begin{prop}     \label{thm:bondo}
  For any $f:S\to\mathbb{R}$,
  \begin{eqnarray*}
  \E\left[\ip{f}{\nu_t^N} \right] &=& \langle
  f,\mu_t\rangle ,\qquad\mbox{ and}\\
  \E\left[\big|\ip{f}{\nu_t^N} - \ip{f}{\mu_t}\big|^2\right] & =& \frac1N \var_{\mu_t}(f)
    + \frac1N \int_0^t \E\left[ V_{s,t}^N(f)\right]\,ds\, ,
  \end{eqnarray*}
 where
\begin{align}
    V_{s,t}^N(f) = & -\ip{H_s(q_{s,t}f)^2}{\nu_s^N} \ip{1}{\nu_s^N}
       - \ip{H_s}{\nu_s^N} \ip{q_{s,t}f^2-(q_{s,t}f)^2}{\nu_s^N}\nonumber\\
       & + \frac 12 \iint |H_s(z)-H_s(y)| (q_{s,t}f(z)-q_{s,t}f(y))^2
           \,\nu_s^N(dy)\,\nu_s^N(dz).     \label{eq:VN}
  \end{align}
\end{prop}
Here and in the following $\var_\mu(f):=\ip{f^2}{\mu}-\ip{f}{\mu}^2$
stands for the variance of $f$ with respect to the measure $\mu$.
Although the reasoning is very close to \cite{MM}, a complete proof of
Proposition \ref{thm:bondo} is given in Section \ref{sec:mp} below for
the reader's convenience.

\smallskip

Elementary estimates show that the approximation error
(\ref{eq:ERR}) for estimates based on the empirical distributions
$\eta_t^N$ can be controlled by the variance of estimators based on
$\nu_t^N$:
\begin{lemma}    \label{lm:nueta}
  For all functions $f:S \to \erre$ and $t \geq 0$ we have
  \begin{equation}     \label{eq:L2L2}
    \E \left[\left| \ip{f}{\eta_t^N}-\ip{f}{\mu_t} \right|^2\right]
    \leq 2 \var \left(\ip{f}{\nu_t^N}\right)\, +\,
    2 \|f-\ip{f}{\mu_t}\|_{\sup }^2 \var\left(\ip{1}{\nu_t^N}\right)
  \end{equation}
  and
  \begin{align}     \label{eq:L1L2}
    \E \left[\left| \ip{f}{\eta_t^N}-\ip{f}{\mu_t} \right|\right]
    &\leq \var \left(\ip{f}{\nu_t^N} \right)^{1/2} +
    \sqrt{2} \|f-\ip{f}{\mu_t}\|_{\sup }
    \var\left(\ip{1}{\nu_t^N}\right)\\
    &\quad + \sqrt 2 \,\var \left(\ip{f}{\nu_t^N} \right)^{1/2}
     \var \left(\ip{1}{\nu_t^N} \right)^{1/2},\nonumber
  \end{align}
  where $\|g\|_{\sup}:= \sup_{x\in S} |g(x)|$ for any $g:S \to \erre$.
\end{lemma}
The proof is given in Section \ref{sec:p2} below.
\begin{rmk}
  A very interesting alternative expression for the variance of
  normalizing constants similar to $\ip{1}{\nu_t^N}$ in discrete time
  has recently been derived in \cite{cerone}.
\end{rmk}

\subsection{A quantitative variance bound}
Let $p \in [2,\infty[$. Our goal is to prove quantitative bounds for
the approximation errors that hold uniformly for all functions
$f:S \to \erre$ with $L^p$ norm less than one.  Because of Lemma
\ref{lm:nueta}, the errors can be quantified in terms of the
variance bounds
\begin{equation}
  \label{eq:ETNP}
  \varepsilon_t^{N,p} := \sup\, \left\{\left.
  \E \left[\big| \ip{f}{\nu_s^N}-\ip{f}{\mu_s} \big|^2\right]\,\right|
  \; f:S\to\erre \text{ s.t. }
  \|f\|_{L^p(\mu_s)} \leq 1, \; s\in[0,t] \right\}
\end{equation}
with $p\in [2,\infty)$.
To efficiently bound the quantities $\varepsilon_t^{N,p}$ we apply
estimates of $L^p$-$L^q$ operator norms for the operators $q_{s,t}$.
Corresponding estimates are derived systematically in \cite{JMAA}.  We
first state a general result that bounds the error in terms of the
expression (\ref{eq:211}) and appropriate operator norms, see Theorem
\ref{thm:main} below.

\smallskip

For $p$, $q \in [2,\infty]$ with $p \leq q$, let us consider the
operator norms
\begin{align*}
  C_{s,t}(p) &:= \sup_{f \neq 0} \frac{\|q_{s,t}f\|_{L^p(\mu_s)}}{%
                                    \|f\|_{L^p(\mu_t)}},\\
  C_{s,t}(p,q) &:= \sup_{f \neq 0} \frac{\|q_{s,t}f\|_{L^{2r}(\mu_s)}}{%
                                    \|f\|_{L^p(\mu_t)}}
    \vee \sup_{f \neq 0} \frac{\|q_{s,t}f\|_{L^p(\mu_s)}}{\|f\|_{L^{p/2}(\mu_t)}}
    \vee 1,
\end{align*}
where $r \in [p,\infty]$ is chosen such that $p^{-1}=q^{-1}+r^{-1}$.
Moreover, for $\delta>0$, we set
\begin{equation*}
  \bar{C}_t(p,q,\delta) := \sup_{\tau \in [0,t]}
  \int_0^{(\tau-\delta)^+} \|H_s\|_{L^q(\mu_s)} C_{s,\tau}(p,q)^2\,ds.
\end{equation*}
We fix a constant $t_0>0$, and set
\begin{equation}
  \label{eq:119}
  \omega := \sup_{s \in [0,t_0]} \osc(H_s),
\end{equation}
where $\osc(f):=\sup f - \inf f$.
Since $H_s=-\frac{\partial}{\partial s} \log\mu_s$, the constant
$\omega$ controls the logarithmic time change rate of the measures
$\mu_t$. Note that
\[
\bar{C}_t(p,q,\delta) \leq t\,\omega \sup\big\{ C_{s,\tau}(p,q)^2\,\big|\,
s,\tau \in [0,t] \text{ s.t. } \tau \geq s+\delta \big\}.
\]
\begin{rmk}
  Since we assume that the state space is finite, all the constants
  are finite, but their numerical values can be very large.  It is a
  straightforward consequence of the forward equation (\ref{eq:FWE})
  that
  \begin{equation}
    \label{eq:INV}
    \mu_s q_{s,t} = \mu_t, \qquad 0 \leq s \leq t,
  \end{equation}
  and hence $C_{s,t}(1) = 1$. On the other hand, in contrast to Markov
  transition operators which are contractions on $L^\infty$, the
  constants $C_{s,t}(\infty)$ can be extremely large in typical
  applications. Therefore bounds on $C_{s,t}(p)$ are very sensitive to
  the choice of $p$, see \cite{JMAA} for details. The constants
  $C_{s,t}(p,q)$ and $\bar C_t(p,q,\delta )$ are related to
  hyperbound properties and can only be expected to be bounded
  in a feasible way if $t-s$ and $\delta$, respectively, are not too
  small.
\end{rmk}

For a function $f:S\to\mathbb{R}$, set
\begin{equation}\label{eq:V}
   V_{s,t}(f) := -\ip{H_s(q_{s,t}f)^2}{\mu_s}
        + \iint |H_s(x)| (q_{s,t}f(y)-q_{s,t}f(x))^2
           \,\mu_s(dx)\,\mu_s(dy).
\end{equation}
Our first main result shows that for $p>4$ the asymptotic (as $N
\to \infty$) variance of the estimator $\ip{f}{\nu_t^N}$ is
bounded from above by
\begin{equation}     \label{eq:211}
N^{-1} \, \Big( \var_{\mu_t} (f) + \int_0^t V_{s,t}(f)\,ds +
\|f\|^2_{L^p(\mu_t)} \Big),
\end{equation}
and, more importantly, it gives a non-asymptotic bound for the mean
square error $\var \left(\ip{f}{\nu_t^N}\right)$ of the same order:
\begin{thm}     \label{thm:main}
  Fix $q\in ]6,\infty]$ and $p \in
  ]\frac{4q}{q-2},q[$. Let $N \in \mathbb{N}$ be such that
  \[
  N \geq 25 \max \big( 2, \bar{C}_{t_0}(p,q,\delta),
  \bar{C}_{t_0}(\tilde{p},q,\delta) \big),
  \]
  where $\tilde{p}$ is defined by $\tilde{p}^{-1}=q^{-1}+(p/2)^{-1}$
  and $\delta:= (17\omega)^{-1}$.
  Then, for $t\in [0,t_0]$,
  \begin{equation}     \label{eq:main1}
  \begin{split}
    N \, \E\left[ |\ip{f}{\nu_t^N} - \ip{f}{\mu_t}|^2\right]\ \leq\
    & \var_{\mu_t}(f)
    + \int_0^t V_{s,t}(f)\,ds\\
    & + \left[ 1+7\bar{C}_t(p,q,\delta) \varepsilon_t^{N,p}\right]
        \,\|f\|^2_{L^p(\mu_t)}.
  \end{split}
  \end{equation}
  In particular,
  \begin{equation}     \label{eq:main2}
    \varepsilon_t^{N,p} \leq (2+v_t(p)) N^{-1}
    \big(1+10\bar{C}_t(p,q,\delta) N^{-1}\big)
  \end{equation}
  where
  \[
  v_t(p) := \sup_{\tau\in[0,t]} \sup_{f \neq 0} \frac{\int_0^\tau
    V_{s,\tau}(f)\,ds}{\|f\|^2_{L^p(\mu_\tau)}}.
  \]
\end{thm}
The proof is given in Section \ref{sec:p1} below.
To apply Theorem \ref{thm:main} we need bounds for the constants
$v_t(p)$ and $\bar{C}_t(p,q,\delta )$. We will now discuss how to
derive such bounds from Poincar\'{e} and logarithmic Sobolev
inequalities in the following particular cases:
\begin{itemize}
\item[a)] The Markov processes with generators $\mathcal{L}_t$, $t\ge
  0$, have ``good'' global mixing properties (see \S\ref{sect14}).
\item[b)] The state space $S$ can be decomposed into disjoint subsets
  $S_i$, $i\in I$, such that $\mathcal{L}_t(x,y)=0$ for all $t\ge 0$,
  $x\in S_i$ and $y\in S_j$ with $i\neq j$, and ``good'' mixing
  properties hold on each of the subsets $S_i$ (see \S\ref{sect15}).
\end{itemize}

\subsection{Non-asymptotic bounds from global Poincar\'{e} and log
Sobolev inequalities}\label{sect14}
For $t \geq 0$ and $q\in[1,\infty]$ let us define
\[
K_t(q) = \int_0^t \|H_s\|_{L^q(\mu_s)}\,ds.
\]
The quantities $K_t(q)$ are a way to control how much the measures
$\mu_s$ change for $s\in[0,t]$. A rough estimate yields
\begin{align}
  v_t(p) &\leq 5 K_t(2) \sup \big\{
  C_{s,\tau}(4)^2 \,\big|\, 0 \leq s \leq \tau \leq t \big\}&
  &\text{for any } p\geq 4, \label{eq:stella}\\
  \bar{C}_t(p,q,\delta )  &\leq K_t(q) \sup\big\{
  C_{s,\tau}(p,q)^2 \,\big|\, 0\leq s \leq s+\delta \leq \tau \leq t \big\}&
  &\text{for any } q\geq p\geq 1.  \label{eq:stst}
\end{align}
Hence estimates for $v_t(p)$ and $ \bar{C}_t(p,q,\delta )$ follow
from appropriate $L^p$-$L^q$ bounds for the Feynman-Kac propagators
$q_{s,t}$. In \cite{JMAA}, we derive such bounds systematically from
Poincar\'e and logarithmic Sobolev inequalities. To apply these
results let us define the weighted Poincar\'e and log Sobolev
constants
\begin{align*}
  A_t &:= \sup_{f\in\mathcal{S}_0} \frac{-\int H_t f^2\,d\mu_t}{\EE_t(f)},\\
  B_t &:= \sup_{f\in\mathcal{S}_0}
               \frac{\big|\int H_t f\,d\mu_t\big|^2}{\EE_t(f)},\\
  \gamma_t &:= \sup_{f\in\mathcal{S}_1}
                   \frac{\int f^2\,\log |f|\,d\mu_t}{\EE_t(f)},
\end{align*}
where $\mathcal{S}_0 = \big\{ f:S \to \erre |\; \ip{f}{\mu_t}=0, \;
f\not\equiv 0 \big\}$,
$\mathcal{S}_1 = \big\{ f:S \to \erre |\; \ip{f^2}{\mu_t}=1, \;
f\not\equiv 1 \big\}$,
and
\[
\EE_t(f)\ =\ - \int f \L_tf\,d\mu_t\ =\ \frac12 \sum_{x,y\in S}
(f(y)-f(x))^2\L_t(x,y)\mu_t(x)
\]
denotes the Dirichlet form of the self-adjoint operator $\L_t$ on
$L^2(S,\mu_t)$. We refer to \cite{SC} for background on Poincar\'e
and logarithmic Sobolev inequalities and their applications to
estimate $L^p$ contractivity properties of transition semigroups and
mixing times of reversible time-homogeneous Markov chains. In
\cite{JMAA} we apply similar techniques to derive $L^p$-$L^q$ bounds
for Feynman-Kac propagators. We show that $C_{s,t}(p)$ and
$C_{s,t}(p,q)$ are small (in particular less than 2) if the
intensities $\lambda_s$, $0 \leq s \leq t$, of MCMC moves are
sufficiently large in terms of the constants $A_s$, $B_s$ and
$\gamma_s$, respectively. By combining these results with Theorem
\ref{thm:main} we obtain:
\begin{thm}     \label{thm:main2}
  Fix $t_0 \geq 0$, $q\in ]6,\infty [$ and $p \in ]\frac{4q}{q-2},q[$.
  Suppose that
  \begin{equation}    \label{eq:Nr}
    N \geq 40 \, \max (K_{t_0}(q),1), \qquad\text{and}
  \end{equation}
  \begin{equation}    \label{eq:2ast}
    \lambda_s  \geq \max\,\left( \frac{pA_s}4 + \frac{p(p+3)}{4}\, t_0B_s\; ,\
    \frac{17}{4} a(p,q)\,\omega\gamma_s\right)
    \qquad \text{for all } s\in [0,t_0],
  \end{equation}
  where $\omega$ is defined by (\ref{eq:119}) and
  \[
  a(p,q) := \log \max
  \left(\frac{2r-1}{p-1},\frac{2p-2}{p-2},\frac{p-1}{\tilde{p}-1},
    \frac{2\tilde{p}-2}{\tilde{p}-2}\right),
  \]
  with $\tilde{p}$ and $r$ determined by $\tilde{p}^{-1}=q^{-1}+2p^{-1}$
  and $p^{-1}=q^{-1}+r^{-1}$. Then, for $t\in [0,t_0]$,
  \begin{equation}
    \label{eq:main3}
    \varepsilon_t^{N,p} \ \leq\ (2+8\, K_t(2)) \, N^{-1}
    \big(1+16K_t(q)N^{-1}\big).
  \end{equation}
\end{thm}
Note that the assumptions on $p$ and $q$ guarantee that $\tilde{p}>2$,
so that $a(p,q)$ is finite. The proof of the theorem is given in
Section \ref{sec:p2} below.
\begin{rmk}     \label{rmk:main2}
  (i) The theorem shows that if the intensities $\lambda_s$ are large
  enough, then already a limited number of particles/replicas suffices
  to obtain reasonable error bounds. In particular, if (\ref{eq:2ast})
  holds, then, by (\ref{eq:main3}), a number
  \[
  N \geq \frac{3+10K_t(q)}{\alpha}
  \]
  of particles guarantees $\varepsilon_t^{N,p} \leq \alpha$
  for a given $\alpha \in ]0,1/8[$. In particular, as $\alpha \to 0$,
  a number of particles of order $O(K_t(q)/\alpha)$ is
  sufficient to bound the error by $\alpha$.

  \smallskip

  (ii) Rough bounds for the constants $K_t(q)$, $A_t$ and
  $B_t$ for $t \in [0,t_0]$ are given by
  \[
  K_t(q) \leq t\omega,
  \qquad A_t \leq C^{\mathrm{Poi}}_t\,\max H_t^-, 
  \qquad B_t \leq C^{\mathrm{Poi}}_t\,\var_{\mu_t}(H_t),
  \]
  where $\omega$ is defined by (\ref{eq:119}) and
  \[
  C^{\mathrm{Poi}}_t:= \sup_{f\in\mathcal{S}_0} \frac{\int
    f^2\,d\mu_t}{\EE_t(f)}
  \]
  denotes the Poincar\'e constant, i.e., the inverse spectral gap of
  the generator $\mathcal{L}_t$.  Therefore, assumptions (\ref{eq:Nr})
  and (\ref{eq:2ast}) in Theorem \ref{thm:main2} are satisfied if
  \begin{equation}
    \label{eq:CN}
    N \geq 40 \max (t_0\omega,1)
  \end{equation}
  and
  \begin{equation}
    \label{eq:Clambda}
    \lambda_s \geq \max\Big( \frac{p}{4}
    \big(\max H_s^- + t_0(p+3)\var_{\mu_s}(H_s)\big) C^{\mathrm{Poi}}_s,
    \frac{17}{4} a(p,q) \omega \gamma_s \Big).
  \end{equation}
\end{rmk}
Theorems \ref{thm:main} and \ref{thm:main2} provide non-asymptotic
bounds on the variances of the Monte Carlo estimators
$\ip{f}{\nu_t^N}$ that hold uniformly over all functions $f \in
L^p(\mu_t)$. One can combine these bounds with (\ref{eq:main1}) and
(\ref{eq:L1L2}) to obtain more precise non-asymptotic error bounds for
the Monte Carlo estimators $\ip{f}{\nu_t^N}$ and $\ip{f}{\eta_t^N}$
for a fixed function $f$:
\begin{coroll}     \label{cor:estf}
  Suppose that the assumptions of Theorem \ref{thm:main2} hold, and
  let $f \in L^p(\mu_t)$. Then
  \begin{align*}
  N \, \E \left[\big| \ip{f}{\nu_t^N} - \ip{f}{\mu_t} \big|^2\right]\ \leq\ &
  \var_{\mu_t} (f) + \int_0^t V_{s,t}(f)\,ds + \|f\|^2_{L^p(\mu_t)}
  + R(t,N) \, \|f\|^2_{L^p(\mu_t)},\\
  N^{1/2} \, \E\left[ \big| \ip{f}{\eta_t^N} - \ip{f}{\mu_t} \big|\right]\ \leq\ &
  \left( \var_{\mu_t} (f) + \int_0^t V_{s,t}(f)\,ds
  + \|f-\ip{f}{\mu_t}\|^2_{L^p(\mu_t)} \right)^{1/2}\\
  &\quad + \tilde{R}(t,N) \, \|f-\ip{f}{\mu_t}\|_{\sup }
  \end{align*}
  with explicit constants $R(t,N)$ of order $O(N^{-1})$ and
  $\tilde{R}(t,N)$ of order $O(N^{-1/2})$.
\end{coroll}
The proof is given in Section \ref{sec:p2} below. 

\subsection{Scope and Examples}
Summarizing our results, we make the following observations: the
derived error bounds of a given size for the particle system
approximation rely on the following quantities:
\begin{itemize}
\item[(i)] A uniform upper bound on the oscillations of the
  logarithmic time derivatives $H_t=-\frac{\partial}{\partial t} \log
  \mu_t$.
\item[(ii)] A minimal intensity $\lambda_t$ of MCMC moves. A lower
  bound for the required intensity can be given in terms of the
  constants $A_t$, $B_t$ and $\gamma_t$, or alternatively in terms of
  $\omega$, $C_t^{\mathrm{Poi}}$ and $\gamma_t$.
\item[(iii)] A minimal number $N$ of particles. On a time interval of
  length $t_0$, a number of particles of order $O(\omega t_0
  \alpha^{-1})$ is sufficient to bound the error
  $\varepsilon_{t_0}^{N,p}$ by $\alpha$ (provided $\lambda_t$ is large
  enough).
\end{itemize}

We now illustrate range and limits of applicability of the results in
two examples. The first is a simple one-dimensional example, while the
second discusses the dimensional dependence of the estimates in the
case of product measures.

\subsection*{Example 1. Moving Gaussians -- one dimensional case}
Suppose that $S=\{a,a+1,\ldots,a+\Delta-1\}$ for some $a \in \mathbb{Z}$
and $\Delta \in \mathbb{N}$, and $(\mu_t)_{t \geq 0}$ are probability
measures on $S$ such that
\[
\mu_t(x) \propto \exp\Big( -\frac{(x-m_t)^2}{2\sigma_t^2} \Big),
\qquad
x \in S.
\]
We assume that $t \mapsto m_t$ and $t \mapsto \sigma_t$ are
continuously differentiable functions such that $\sigma_t \in
]0,\infty[$ and $m_t \in [a,a+\Delta-1]$ for all $t \geq 0$. Moreover,
we assume that the Markov chain moves are given by a Random Walk
Metropolis dynamics (in continuous time), that is,
\[
\mathcal{L}_t(x,y) = 
\begin{cases}
\frac12 \min\big( \frac{\mu_t(y)}{\mu_t(x)},1\big), &\text{if } |y-x|=1,\\
0, &\text{if } |y-x| \geq 1.
\end{cases}
\]
In this case, the following upper bounds for $C_t^{\mathrm{Poi}}$ and
$\gamma_t$ hold (see the Appendix):
\begin{align}
  \label{eq:B2}
  C_t^{\mathrm{Poi}} &\leq 30\big( (\sigma_t \wedge \Delta) \vee 2\big)^2\\
  \label{eq:B3}
  \gamma_t &\leq 300 \frac{\Delta^2}{(\sigma_t\wedge 1)^2}
                 + 300 \big( (\sigma_t \wedge \Delta) \vee 2\big)^2
                   \log \Delta
\end{align}
It can be shown that the upper bound for $C_t^{\mathrm{Poi}}$ is of
the correct order in $\sigma_t$ and $\Delta$. The upper bound for
$\gamma_t$ could be improved, but $\gamma_t$ is always bounded from
below by a positive multiple of $(\Delta/\sigma_t)^2$. Our results can
be applied in the following way.  For $t \geq 0$ and $x$, $y \in
S$ we have
\begin{align}
\nonumber
H_t(x)-H_t(y) &= \frac{\partial}{\partial t} \Big(
\frac{(x-m_t)^2}{\sigma_t^2} - \frac{(y-m_t)^2}{\sigma_t^2}
\Big)\\
\nonumber
&= -\frac{\sigma_t'}{\sigma_t} \frac{(x-y)(x+y-2m_t)}{\sigma_t^2}
-m'_t \frac{x-y}{\sigma_t^2}\\
\label{eq:B1}
&\leq \Big( 2\frac{|\sigma_t'|}{\sigma_t} + \frac{|m_t'|}{\Delta}
      \Big) \frac{\Delta^2}{\sigma_t^2}.
\end{align}
Therefore, if we choose the time scale in such a way that the condition
\begin{equation}
  \label{eq:A1}
  2\frac{|\sigma_t'|}{\sigma_t} + \frac{|m_t'|}{\Delta}
  \leq \frac{\sigma_t^2}{\Delta^2} \qquad \forall t \in [0,t_0]
\end{equation}
is satisfied, then
\[
\omega = \sup_{t\in[0,t_0]} \osc(H_t) \leq 1.
\]

\medskip

Condition (\ref{eq:A1}) is an upper bound on the relative change rates
of the parameters $\sigma_t$ and $m_t$. Note that if $\Delta$ is large
compared to $\sigma_t$, then only small change rates are possible. The
reason is that in this case the Gaussian measure $\mu_t$ changes too
rapidly in the tails, so that our arguments break down.

Assuming (\ref{eq:A1}), Theorem \ref{thm:main2} and Remark
\ref{rmk:main2}(iii) imply that
\[
\varepsilon_t^{N,p} \leq (2+8t) N^{-1} (1+16N^{-1}),
\]
provided $N \geq 40 (t_0 \vee 1)$, and (\ref{eq:Clambda}) holds with
$\omega=1$, $\max H_s^-$ and $\var_{\mu_s}(H_s)$ bounded by
$\omega=1$, and $C_s^{\mathrm{Poi}}$, $\gamma_s$ replaced by the upper
bounds in (\ref{eq:B2}), (\ref{eq:B3}). If $(\sigma_t \wedge
1)/\Delta$ is not too small, this yields reasonably sized (although
far from optimal) lower bounds on $\lambda_t$ and $N$. On the other
hand, if $\sigma_t/\Delta \to 0$, then the upper bounds in both
(\ref{eq:B1}) and (\ref{eq:B3}) degenerate drastically.

\subsection*{Example 2. Product measures -- dependence on the dimension}
In our second example we study the dependence of (\ref{eq:CN}),
(\ref{eq:Clambda}) and (\ref{eq:main3}) on the dimension in the case
when the evolving measures are all product measures. Suppose that
\[
S = \prod_{i=1}^d S_i,
\qquad \mu_t = \bigotimes_{i=1}^d \mu_t^{(i)},
\]
with probability measures $\mu_t^{(i)}$, $t \geq 0$, $i=1,\ldots,d$,
on finite sets $S_i$ such that $t \mapsto \mu_t^{(i)}(x)$ is
continuously differentiable and strictly positive for all $1\leq i
\leq d$ and $x \in S_i$. In this case one has
\[
H_t(x) = \sum_{i=1}^d H_t^{(i)}(x_i),
\]
where $H_t$ and $H_t^{(i)}$ denote the negative logarithmic time
derivatives of the measures $\mu_t$ and $\mu_t^{(i)}$, respectively.
If we assume
\[
\osc(H_t^{(i)}) \leq 1 \qquad \forall t \in [0,t_0], \; i=1,\ldots,d,
\]
then
\begin{equation}
  \label{eq:bound-om}
  \omega = \sup_{t \in [0,t_0]} \osc(H_t) \leq d,
\end{equation}
and
\[
\var_{\mu_t}(H_t) = \sum_{i=1}^d \var_{\mu_t^{(i)}}(H_t^{(i)}) \leq d.
\]
Now suppose that
\[
\mathcal{L}_t(x,y) = \sum_{i=1}^d \mathcal{L}_t^{(i)}(x_i,y_i)
\]
for generators $\mathcal{L}_t^{(i)}$, $t \geq 0$, $i=1,\ldots,d$, of
time-inhomogeneous Markov processes on $S_i$, i.e. $\mathcal{L}_t$ is
the generator of the product dynamics on $S$ with component generators
$\mathcal{L}_t^{(i)}$. It is well known that $\mathcal{L}_t$ satisfies
Poincar\'e and logarithmic Sobolev inequalities with constants
\[
C_t^{\mathrm{Poi}} = \max_{i=1,\ldots,d} C_t^{\mathrm{Poi},(i)},
\qquad
\gamma_t = \max_{i=1,\ldots,d} \gamma_t^{(i)},
\]
respectively, where $C_t^{\mathrm{Poi},(i)}$ and $\gamma_t^{(i)}$ are
the Poincar\'e and logarithmic Sobolev constants for the generators
$\mathcal{L}_t^{(i)}$. In particular, if the component generators
$\mathcal{L}_t^{(i)}$ satisfy Poincar\'e and logarithmic Sobolev
inequalities with constants independent of $i$, then $\mathcal{L}_t$
satisfies the corresponding inequalities with the same constants --
independently of the dimension $d$. Therefore, in this case, the
values of $N$ and $\lambda_s$ required to satisfy conditions
(\ref{eq:CN}) and (\ref{eq:Clambda}) are of order
$O(d)$. Hence both the number of particles/replicas and the intensity
of MCMC moves required are of order $O(d)$. Since simulating from the
product dynamics also requires $O(d)$ steps, the total effort to keep
track of the evolving product measures up to a given precision is of
order $O(d^3)$.

\begin{rmk}[Independent particles]
  We compare briefly with the particle dynamics without importance
  sampling/resampling, i.e., when the second summand is omitted in the
  definition (\ref{eq:generator}) of the generator
  $\mathcal{L}_t^N$. In this case, the particles/replicas move
  independently according to the time-inhomogeneous Markovian dynamics
  with generators $\mathcal{L}_t$, $t \geq 0$.  Hence the positions of
  the particles at time $t$ are independent random variables with
  distribution $\tilde{\mu}_t=\mu_0 p_{0,t}$, where $p_{s,t}$, $0 \leq
  s \leq t$, is the time-inhomogeneous transition function. A
  corresponding discrete-time dynamics is used for example in the
  classical simulated annealing algorithm (see
  e.g. \cite{DMM-ann,liu}). Since in general $\tilde{\mu}_t \neq
  \mu_t$, the empirical distribution of the independent particle
  system is an asymptotically biased estimator for $\mu_t$. However,
  under strong mixing conditions as imposed above, the difference
  between $\tilde{\mu}_t$ and $\mu_t$, and hence the asymptotic bias,
  will be small. Therefore it is possible that, for fixed $N$, the
  empirical distribution of the independent particles process is a
  better estimate for $\mu_t$ than $\eta_t^N$. On the other hand, if
  the mixing properties break down, the bias of the independent
  particles estimator will not be small, whereas the empirical
  measures $\nu_t^N$ and $\eta_t^N$ may still be suitable
  estimators. This will be demonstrated now in a particular case.
\end{rmk}

\subsection{Non-asymptotic bounds from local
estimates}\label{sect15}
With suitable modifications the above analysis can also be applied to
derive bounds when good mixing properties hold only locally. As an
illustration, we consider another extreme case in which the state
space is decomposed into several components that are not connected by
the underlying Markovian dynamics.
Suppose that
\[
S = \bigcup_{i\in I} S_i,
\]
is a decomposition of $S$ into disjoint non-empty subsets $S_i$,
$i\in I$, such that
\[
\L_t(x,y)\ =\ 0\qquad\mbox{for any }t \geq 0,\
x\in S_i\mbox{ and } y \in S_j\mbox{ with } i \neq j.
\]
Let $\mu_t^i:=\mu_t(\cdot|S_i)$ denote the measure $\mu_t$ conditioned
by $S_i$. Then we can apply the arguments above with the $L^p$ norm
replaced by the stronger norm
\[
\| f \|_{L^p(\mu_t)}^\sim\ :=\ \max_{i
  \in I} \|f\|_{L^p(S_i,\mu_t^i)}.
\]
Since H\"older's inequality and related estimates hold for these
modified $L^p$ norms as well, the assertion of Theorem
\ref{thm:main} still remains true if $\varepsilon_t^{N,p}$ is
replaced by
\[
\tilde{\varepsilon}_t^{N,p} :=
\sup\, \left\{\left.
\E \left[\big| \ip{f}{\nu_s^N}-\ip{f}{\mu_s} \big|^2\right]\,\right|
\; f:S\to\erre \text{ s.t. }
\|f\|_{L^p(\mu_s)}^\sim \leq 1, \; s\in[0,t] \right\},
\]
and the constants $C_{s,t}(p,q)$ and $\bar{C}_t(p,q,\delta )$ are
defined w.r.t. the modified $L^p$ and $L^q$ norms as well. Moreover,
the representations (\ref{eq:timedep}) and (\ref{eq:centered}) hold for
$\mu_t^i$ in place of $\mu_t$ if $H_t$ is replaced by
\[
H_t^i\ :=\ H_t-\langle H_t,\mu_t^i\rangle\, .
\]
Let $A_t^i$, $B_t^i$ and $\gamma_t^i$ denote the Poincar\'{e} and
logarithmic Sobolev constants defined as above but with $S$, $\mu_t$
and $H_t$ replaced by $S_i$, $\mu_t^i$ and $H_t^i$, respectively. Let
us also set
\[
\tilde A_t:=\max_{i\in I}\, A_t^i,\qquad \tilde
B_t:=\max_{i\in I}\, B_t^i,\qquad \tilde\gamma_t:=\max_{i\in
  I}\,\gamma_t^i,
\]
\begin{align*}
  \tilde{K}_t(q) &\ :=\ \int_0^t \| H_s\|_{L^q(\mu_s)}^\sim \,ds,\qquad\mbox{ and}\\
  \tilde{M}_t &\ :=\ \max_{i \in I} \sup_{0 \leq r \leq s \leq t}
                    \,\frac{\mu_s(S_i)}{\mu_r(S_i)}\ .
\end{align*}
Then, by estimating $L^p$ norms separately on each component, we
can prove the following extension of Theorem \ref{thm:main2}:
\begin{thm}     \label{thm:main3}
  Fix $t_0 \geq 0$, $q\in ]6,\infty [$ and $p \in
  ]\frac{4q}{q-2},q[$. Suppose that
  \[
  N \geq 40 \max (\tilde K_{t_0}(q),1), \qquad\text{and}
  \]
  \begin{equation}
    \label{eq:3ast}
    \lambda_s \geq \max \left(
    \frac{p\tilde A_s}4 + \frac{p(p+3)}{4}t_0\tilde B_s,
    \frac{17}{4} a(p,q) \omega \tilde{\gamma}_s\right)
    \qquad \text{for all } s\in [0,t_0].
  \end{equation}
  Then, for $t\in [0,t_0]$, one has
  \[
  \tilde{\varepsilon}_t^{N,p} \ \leq\ \big(2+8K_t(2)
  \tilde{M}^2_t\big)\, N^{-1}\, \big(1 + 16\tilde{K}_t(q)\tilde{M}_t^2\,
  N^{-1}\big).
  \]
\end{thm}
\begin{rmk} (i) If there is only one component, the assertion of
  Theorem \ref{thm:main3} reduces to that of Theorem \ref{thm:main2}.

  \noindent (ii) Error bounds for the estimators $\ip{f}{\nu_t^N}$ and
  $\ip{f}{\eta_t^N}$ for a fixed function $f$ hold analogously to
  Corollary \ref{cor:estf}.
\end{rmk}

\subsection{Open problems}
1) The cases discussed in Sections \ref{sect14} and \ref{sect15} are
extreme cases. In many typical applications, one would expect the
state space to split up as time evolves into more and more
components that get almost disconnected by the dynamics (local
modes, metastable states). The study of such more complicated
situations is an important topic for future research.

\smallskip

\noindent 2) We have discussed here a setup with discrete state space
and continuous time. In continuous time, particle systems on more
general state spaces can in principle be treated by similar
techniques, although of course additional technical considerations are
required (cf. for instance \cite{Rou}). For algorithmic applications,
the case of discrete time and a continuous state space is probably the
most interesting one. For an overview of the substantial literature
and some more recent results in this case we refer to
\cite{BCJ,cappe,chop-ann,DM,MM,DMDJ,DouMou,JasDou} and references
therein. An $L^p$ approach similar to the one presented here is
developed for the discrete time case in the PhD thesis of N.~Schweizer
\cite{Schw-phd}.

\section{Variances of weighted empirical averages}
\label{sec:mp}
In this section we will prove Proposition \ref{thm:bondo}, which shows
that $\langle f,\nu_t^N\rangle $ is an unbiased estimator for $\langle
f,\mu_t\rangle $ and gives an explicit formula for the variance.  The
proof follows the arguments developed in \cite{MM} relying on the
identification of appropriate martingales.

Recall that the carr\'{e} du champ (square field) operator
$\Gamma_t^N$ associated to $\mathcal{L}_t^N$ is defined for functions
$\varphi :S^N\to\mathbb{R}$ by
\[
\Gamma_t^N(\varphi )\ =\
\mathcal{L}_t^N\varphi^2-2\varphi\mathcal{L}_t^N\varphi,
\]
i.e.,
\begin{equation}\label{eq:Gamma}
\Gamma_t^N(\varphi )(x)\ =\ \sum_{y\in S}\mathcal{L}_t^N(x,y)\,
(\varphi (y)-\varphi (x))^2\qquad \forall x\in S^N.
\end{equation}
It is well-known that the processes
\begin{eqnarray}
M_t^\varphi &=& \varphi(t,X_t^N) - \varphi(0,X_0^N)
      - \int_0^t (\frac{\partial }{\partial s} +
      \L_s^N)\varphi(s,X_s^N)\,ds,\qquad\mbox{and}\label{eq:Mt}\\
 N_t^\varphi &=& (M_t^\varphi )^2 - \int_0^t \Gamma_s^N(\varphi(s,\,\cdot\,
 ))(X_s^N)\,ds\label{eq:Nt}
\end{eqnarray}
are martingales w.r.t. the filtration induced by the process $X_t^N$
for any function $\varphi :\mathbb{R}^+\times S^N\to\mathbb{R}$ that
is twice continuously differentiable in the first variable, cf. e.g.
\cite[Appendix 1, Lemma 5.1]{KipLan}.  For $x\in S^N$ let
\[
\eta(x) = \frac1N \sum_{i=1}^N \delta_{X_i}
\]
denote the corresponding empirical average. In the next lemma
we derive expressions for $\L_t^N$ and $\Gamma_t^N$ acting on linear
functions on $S^N$ of the form
\[
\varphi_f(x)\ =\ \ip{f}{\eta(x)}\ =\ N^{-1}\sum_{i=1}^Nf(x_i).
\]
\begin{lemma}    \label{lem:mozza}
For any function $f:S\to\mathbb{R}$ and $t\ge 0$, one has
\begin{equation*}
\L_t^N \ip{f}{\eta}\ =\ \lambda_t \ip{\L_t f}{\eta} + \ip{H_t}{\eta}
\ip{f}{\eta} - \ip{H_t f}{\eta}
\end{equation*}
and
\begin{equation*}
\Gamma_t^N (\ip{f}{\eta})\ =\ \frac{\lambda_t}{N} \ip{\Gamma_t
(f)}{\eta} + \frac 1N \iint \big(H_t(y)-H_t(z)\big)^+
\big(f(z)-f(y)\big)^2 \,\eta(dy)\,\eta(dz),
\end{equation*}
where $\Gamma_t$ denotes the carr\'{e} du champ operator w.r.t.
$\mathcal{L}_t$.
\end{lemma}
\begin{proof}
The definition of $\L_t^N$ immediately yields
\begin{equation}
\label{eq:provo} \L_t^N \ip{f}{\eta }(x)\  =\ \frac{\lambda_t}{N}
\sum_{i=1}^N \L_t f(x_i) + \frac1{N^2} \sum_{i,j=1}^N
(H_t(x_i)-H_t(x_j))^+ (f(x_j)-f(x_i)).
\end{equation}
Moreover,
\begin{align*}
&\sum_{i,j=1}^N (H_t(x_i)-H_t(x_j))^+ (f(x_j)-f(x_i))\\
&\hspace*{7em} =
  \sum_{i,j: H_t(x_i)>H_t(x_j)} (H_t(x_i)-H_t(x_j)) (f(x_j)-f(x_i))\\
&\hspace*{7em} =
  \sum_{i,j: H_t(x_j)>H_t(x_i)} (H_t(x_j)-H_t(x_i)) (f(x_i)-f(x_j))\\
&\hspace*{7em} =
  \sum_{i,j: H_t(x_j)>H_t(x_i)} (H_t(x_i)-H_t(x_j)) (f(x_j)-f(x_i))\\
&\hspace*{7em} =
  -\sum_{i,j=1}^N (H_t(x_i)-H_t(x_j))^- (f(x_j)-f(x_i)),
\end{align*}
and hence
\[
\sum_{i,j=1}^N (H_t(x_i)-H_t(x_j)) (f(x_j)-f(x_i)) = 2
\sum_{i,j=1}^N (H_t(x_i)-H_t(x_j))^+ (f(x_j)-f(x_i)).
\]
Therefore the second term on the right hand side of (\ref{eq:provo})
is equal to
\begin{align*}
\frac{1}{2N^2} &\sum_{i,j=1}^N (H_t(x_i)-H_t(x_j)) (f(x_j)-f(x_i))\\
&\qquad = \Big(\frac1N \sum_{i=1}^N H_t(x_i) \Big) \Big(\frac1N
\sum_{j=1}^N f(x_j) \Big)
 - \frac1N \sum_{i=1}^N H_t(x_i) f(x_i)\\
 &\qquad = \ip{H_t}{\eta(x)} \ip{f}{\eta(x)} - \ip{H_tf}{\eta(x)},
\end{align*}
from which the first claim follows.

Furthermore, since
\[
\ip{f}{\eta (x^{i\to j})}-\ip{f}{\eta (x)}\ =\ N^{-1}\,\left(
f(x_j)-f(x_i)\right),
\]
(\ref{eq:Gamma}) and (\ref{eq:generator}) imply
\begin{align*}
\Gamma_t^N \ip{f}{\eta }(x)\ =&\;
\frac{\lambda_t}{N^2} \sum_{i=1}^N \sum_{y\in S} \L_t(x_i,y) (f(y)-f(x_i))^2\\
&+ \frac 1{N^3}\sum_{i,j=1}^N \big(H_t(x_i)-H_t(x_j)\big)^+
\big(f(x_j)-f(x_i)\big)^2,
\end{align*}
from which the second claim follows noting that the first term on the
right hand side of the previous expression is equal to
\[
\frac{\lambda_t}{N^2} \sum_{i=1}^N \Gamma_t (f)(x_i)\ =\
\frac{\lambda_t}{N} \ip{\Gamma_t (f)}{\eta(x)}.\qedhere
\]
\end{proof}

Now let us define
\[
\bar{A}_{s,t}^f\ = \ \ip{q_{s,t}f}{\eta_s^N}\ =\ \frac1N
\sum_{i=1}^N (q_{st}f)(X^N_{s,i})\, .
\]
As a consequence of Lemma \ref{lem:mozza} we obtain:
\begin{prop}     \label{prop:rella}
  The processes $\bar M_u^f$ and $\bar N_u^f$, $u\in[0,t]$, defined by
  \begin{eqnarray*}
    \bar M_u^f &=& \bar{A}_{u,t}^f - \bar{A}_{0,t}^f
              - \int_0^u \ip{H_s}{\eta_s^N}\ip{q_{s,t}f}{\eta_s^N}\,ds,\\
    \bar N_u^f &=& (\bar M_u^f)^2
              - \frac{1}{N} \int_0^u \lambda_s\,\ip{\Gamma_s (q_{s,t}f)}{\eta_s^N}\,ds\\
          &&  - \frac1N \int_0^u\!\!\iint \big(H_s(y)-H_s(z)\big)^+
                   \big(q_{s,t}f(z)-q_{s,t}f(y)\big)^2
                       \,\eta_s^N(dy) \,\eta_s^N(dz) \,ds
  \end{eqnarray*}
  are martingales w.r.t.\ the filtration $\mathcal{F}_t=\sigma (X_s^N\, |\, s\in [0,t])$.
\end{prop}
\begin{proof}
Note that $\bar{A}_s^f=\varphi(s,X_s^N)$, where
\[
\varphi(s,x)\ =\ N^{-1}\sum_{i=1}^N q_{st}f(x_i).
\]
By the backward equation (\ref{eq:BWE}),
\begin{eqnarray*}
 \frac{\partial}{ \partial s}\varphi(s,x)
     &=& -\frac{\lambda_s}{N} \sum_{i=1}^N \L_s q_{st}f(x_i)
         +\frac1N \sum_{i=1}^N H_s q_{st}f(x_i)\\
     &=& -\lambda_s \ip{\L_s q_{st}f}{\eta(x)} + \ip{H_s q_{st}f}{\eta(x)},
\end{eqnarray*}
and by lemma \ref{lem:mozza},
\begin{eqnarray*}
  \left(\L_s^N\varphi\right)(s,x) &=& \lambda_s \ip{\L_sq_{s,t}f}{\eta(x)}
                         + \ip{H_s}{\eta(x)}\ip{q_{s,t}f}{\eta(x)}
                         - \ip{H_sq_{s,t}f}{\eta(x)}
\end{eqnarray*}
Hence
\[
\left(\frac{\partial}{\partial s} + \L_s^N\right)\varphi(s,x)\ =\
\ip{H_s}{\eta(x)}\ip{q_{s,t}f}{\eta(x)},
\]
which proves that $\bar M^f=M^\varphi $ is a martingale, cf.\
(\ref{eq:Mt}). Similarly, by Lemma \ref{lem:mozza},
\begin{eqnarray*}
  \Gamma_s^N(\varphi )(s,x) &=& \frac{\lambda_s}{N} \ip{\Gamma_s (q_{s,t}f)}{\eta(x)}\\
      && + \frac1N \iint \big(H_s(y)-H_s(z)\big)^+
                   \big(q_{s,t}f(z)-q_{s,t}f(y)\big)^2
                       \,\eta_s^N(dy) \,\eta_s^N(dz),
\end{eqnarray*}
which proves that $\bar N^f=N^\varphi $ is a martingale, cf. (\ref{eq:Nt}).
\end{proof}

Since in general, $\bar A^f_{s,t}$ is not a martingale,
$\ip{f}{\eta_t^N}$ is not an unbiased estimator for $\ip{f}{\mu_t}$.
This motivates considering $\ip{f}{\nu_t^N}$ instead. Let
\begin{equation}    \label{eq:lone}
A_{s,t}^f \ =\ \ip{q_{s,t}f}{\nu_s^N}\ =\ e^{-\int_0^s
\ip{H_r}{\eta_r^N}\,dr} \bar{A}_{s,t}^f.
\end{equation}
\begin{prop}      \label{lem:asiago}
  The process $A_{u,t}^f$, $u\in [0,t]$, is a martingale
  with increasing process given by
  \begin{eqnarray*}
    \langle A^f_{\bullet ,t} \rangle_u &=& \frac{1}{N} \int_0^u
         \lambda_s\,\ip{1}{\nu_s^N} \ip{\Gamma_s (q_{s,t}f)}{\nu_s^N}\,ds\\
      && +\frac1N \int_0^u
         \iint \big(H_s(x)-H_s(y)\big)^+
            \big(q_{s,t}f(y)-q_{s,t}f(x)\big)^2
                \,\nu_s^N(dx) \,\nu_s^N(dy)\; ds.
  \end{eqnarray*}
\end{prop}
\begin{proof}
  By the integration by parts formula for Stieltjes integrals and
  Proposition \ref{prop:rella}, we get
  \begin{eqnarray*}
    A_{u,t}^f-A_{0,t}^f &=& \int_0^u e^{-\int_0^s
      \ip{H_r}{\eta_r^N}\,dr} d\bar{A}_{s,t}^f\,
    - \,\int_0^u\ip{H_s}{\eta_s^N} e^{-\int_0^s \ip{H_r}{\eta_r^N}\,dr} \bar{A}_{s,t}^f\,ds\\
    &=& \int_0^u e^{-\int_0^s \ip{H_r}{\eta_r^N}\,dr} d\bar M_s^f +
    \ip{H_s}{\eta_s^N} A_s^f\,ds - \ip{H_s}{\eta_s^N} A_s^f\,ds.
  \end{eqnarray*}
  Hence $[0,t] \ni s \mapsto A_{s,t}^f$ is a martingale whose
  increasing process can be written as
  \[
  \langle A^f_{\bullet ,t} \rangle_u\ =\ \int_0^u e^{-2\int_0^s
    \ip{H_r}{\eta_r^N}\,dr}\,d\langle \bar M^f \rangle_s.
  \]
  The result now follows by Proposition \ref{prop:rella} and Equation
  (\ref{eq:REWEMP}).
\end{proof}

The purpose of the next lemma is to obtain an alternative
representation (modulo martingale terms) of the term involving the
carr\'{e} du champ operator in the expression for $\langle
A_{\bullet ,t}^f\rangle$.
\begin{lemma}    \label{lem:bufala}
  The following decomposition holds:
  \begin{eqnarray*}
    \lefteqn{\int_0^u
      \lambda_s \ip{1}{\nu_s^N}\ip{\Gamma_s (q_{st}f)}{\nu_s^N}\,ds}\\
    &=& \tilde{M}_u
    + \ip{1}{\nu_u^N} \bip{(q_{ut}f)^2}{\nu_u^N}
    + \int_0^u
    \ip{H_s}{\nu_s^N} \ip{(q_{st}f)^2}{\nu_s^N}\,ds\\
    && - \int_0^u \ip{1}{\nu_s^N}
    \ip{H_s(q_{st}f)^2}{\nu_s^N}\,ds,
  \end{eqnarray*}
  where $\tilde{M}$ is a martingale.
\end{lemma}
\begin{proof}
  Let
  \[
  Y_u := \ip{1}{\nu_u^N}\ip{(q_{ut}f)^2}{\nu_u^N} = e^{-2\int_0^u
    \ip{H_r}{\eta_r^N}\,dr} \bip{(q_{ut}f)^2}{\eta_u^N}.
  \]
  By applying the martingale problem to the functions
  $\varphi(s,x)=\bip{(q_{st}f)^2}{\eta(x)}$, we obtain
  \begin{eqnarray*}
    Y_u = e^{-2\int_0^u \ip{H_r}{\eta_r^N}\,dr} \bip{(q_{ut}f)^2}{\eta_u^N}
    &\sim &  - 2 \int_0^u e^{-2\int_0^s \ip{H_r}{\eta_r^N}\,dr}
    \ip{H_s}{\eta_s^N} \bip{(q_{st}f)^2}{\eta_s^N}\,ds\\
    && + \int_0^u e^{-2\int_0^s \ip{H_r}{\eta_r^N}\,dr}
    \left(\frac{\partial}{\partial s} +
      \L_s^N\right)\varphi(s,X_s^N)\,ds.
  \end{eqnarray*}
  Here and in the following we write $Y_u \sim Z_u$ if the processes
  $Y_u$ and $Z_u$ differ only by a martingale term.  Proceeding as in
  the proof of proposition \ref{prop:rella}, we get that
  \[
  \frac{\partial}{\partial s}\varphi(s,X_s^N) \ =\ 2 \ip{q_{st}f
    \frac{\partial}{\partial s} q_{st}f}{\eta_s^N}\ =\
  -2\lambda_s\ip{q_{st}f \L_s q_{st}f}{\eta_s^N} +
  2\ip{H_s(q_{st}f)^2}{\eta_s^N},
  \]
  and
  \[
  \L_s^N\varphi(s,X_s^N) \ =\ \lambda_s \ip{\L_s(q_{st}f)^2}{\eta_s^N}
  + \ip{H_s}{\eta_s^N} \ip{(q_{st}f)^2}{\eta_s^N} -
  \ip{H_s(q_{st}f)^2}{\eta_s^N}.
  \]
  Recalling that $\L_s(q_{st}f)^2 - 2q_{st}f\L_sq_{st}f = \Gamma_s
  (q_{s,t}f)$ and $\nu_s^N=\exp (-\int_0^s\ip{H_r}{\nu_r^N}dr)\,
  \eta_s^N$, we conclude
  \begin{eqnarray*}
    \ip{1}{\nu_u^N}\ip{(q_{ut}f)^2}{\nu_u^N} &\sim &  - \int_0^u
    \ip{H_s}{\nu_s^N} \ip{(q_{st}f)^2}{\nu_s^N}\,ds\\
    && + \int_0^u \ip{1}{\nu_s^N}
    \ip{H_s(q_{st}f)^2}{\nu_s^N}\,ds + \int_0^u
    \lambda_s\, \ip{1}{\nu_s^N}\ip{\Gamma_s
      (q_{st}f)}{\nu_s^N}\,ds,
  \end{eqnarray*}
  which proves the assertion.
\end{proof}

\begin{lemma}\label{lem:26}
For all $t\ge 0$,
\begin{equation*}
\E\Big[ \ip{1}{\nu_t^N} \ip{f^2}{\nu_t^N} \Big] \ =\ \ip{f^2}{\mu_t}
- \E\left[ \int_0^t
    \ip{H_s}{\nu_s^N} \ip{q_{st}f^2}{\nu_s^N}\,ds\right] .
\end{equation*}
\end{lemma}
\begin{proof}
By the product rule for Stieltjes integrals,
\begin{eqnarray*}
    \ip{1}{\nu_s^N} \ip{q_{st}f^2}{\nu_s^N}
& =& e^{-\int_0^s \ip{H_r}{\eta_r^N}\,dr} A^{f^2}_{s,t}
\\
 &=&\int_0^s e^{-\int_0^u \ip{H_r}{\eta_r^N}\,dr} dA^{f^2}_{u,t} \ -\
\int_0^s\ip{H_u}{\nu_u^N} A_{u,t}^{f^2}\,du.
\end{eqnarray*}
Since $s\mapsto A^{f^2}_{s ,t}$ is a martingale,
\[
\E \left[ \ip{1}{\nu_t^N}\ip{f^2}{\nu_t^N}\right]\ =\
\ip{q_{0,t}f^2}{\mu_0}-\E\left[\int_0^t\ip{H_u}{\nu_u^N}A_{u,t}^{f^2}\,du\right]
.
\]
The proof is completed by noting that
$\ip{q_{0,t}f^2}{\mu_0}=\ip{f^2}{\mu_t}$.
\end{proof}

\begin{proof}[Proof of Proposition \ref{thm:bondo}]
 Fix a function
$f:S\to\mathbb{R}$ and $t\ge 0$. Recalling that, by (\ref{eq:INV}),
$\ip{f}{\mu_t}=\ip{q_{0,t}f}{\mu_0}$, we have
\begin{eqnarray*}
\ip{f}{\nu_t^N} - \ip{f}{\mu_t} &=& \ip{q_{t,t}f}{\nu_t^N} -
\ip{q_{0,t}f}{\nu_0^N}
+ \ip{q_{0,t}f}{\nu_0^N} - \ip{q_{0,t}f}{\mu_0}\\
&=& A^f_{t,t} - A_{0,t}^f + \ip{q_{0,t}f}{\nu_0^N} -
\ip{q_{0,t}f}{\mu_0}.
\end{eqnarray*}
Taking expectations on both sides, we immediately obtain
$$\E\left[\ip{f}{\nu_t^N}\right]\ =\ \ip{f}{\mu_t},$$
because $s\mapsto A_{s,t}f$ is a martingale by Proposition
\ref{lem:asiago}, and $\nu_0^N$ is the empirical distribution of $N$
i.i.d.\ random variables with distribution $\mu_0$. Moreover, by
Proposition \ref{lem:asiago} and Lemma \ref{lem:bufala},
\begin{eqnarray*}
\lefteqn{N\,\E\left[\left|\ip{f}{\nu_t^N} -
\ip{f}{\mu_t}\right|^2\right] \ =\
N\,\E\left[(A_{t,t}^f-A_{0,t}^f)^2\right]\, +\,
N\,\E\left[\left(
\ip{q_{0,t}f}{\nu_0^N}-\ip{q_{0,t}f}{\mu_0}\right)^2\right] }\\
&=&N\,\E\left[\langle A_{\bullet ,t}^f\rangle_t\right] +
\var_{\mu_0}(q_{0,t}f)\\
&=&  \E\Big[ \ip{1}{\nu_t^N}
                \ip{f^2}{\nu_t^N}
- \ip{(q_{0,t}f)^2}{\nu_0^N}\Big]\, +\, \var_{\mu_0}(q_{0,t}f)\\
&& + \E\int_0^t
                      \ip{H_s}{\nu_s^N} \ip{(q_{st}f)^2}{\nu_s^N}\,ds\
 -\ \E\int_0^t \ip{1}{\nu_s^N}
                     \ip{H_s(q_{st}f)^2}{\nu_s^N}\,ds\\
&& +  \E \int_0^t
                         \iint \big(H(x)-H(y)\big)^+
                                    \big(q_{s,t}f(y)-q_{s,t}f(x)\big)^2
                                    \,\nu_s^N(dx) \,\nu_s^N(dy)\,ds.
\end{eqnarray*}
The assertion now follows from Lemma \ref{lem:26} observing that
\begin{eqnarray*}
-\E\Big[  \ip{(q_{0,t}f)^2}{\nu_0^N}\Big]\, +\, \var_{\mu_0}(q_{0,t}f) &=&
-\ip{(q_{0,t}f)^2}{\mu_0}\, +\, \var_{\mu_0}(q_{0,t}f)\\
&=&-\ip{q_{0,t}f}{\mu_0}^2 = -\ip{f}{\mu_t}^2.
\end{eqnarray*}
\end{proof}

\section{Proof of Theorem \ref{thm:main}}     \label{sec:p1}
\begin{prop}     \label{prop:long}
Let $p$, $q$, $r \in [1,\infty]$ be such that
$p^{-1}=q^{-1}+r^{-1}$. Then, for $0 \leq s \leq t$,
\begin{align*}
  \E \left[V_{s,t}^N(f)\right] \leq &V_{s,t}(f)\\
  &\quad + \big( 6\|H_s\|_{L^q(\mu_s)} \|q_{s,t}f\|^2_{L^{2r}(\mu_s)}
       + \|H_s\|_{L^p(\mu_s)} \|q_{s,t}f^2\|_{L^p(\mu_s)} \big)
       \varepsilon_s^{N,p}.
\end{align*}
\end{prop}
\begin{proof}
  Since $\ip{f}{\nu_s^N}$ and $\ip{g}{\nu_s^N}$ are unbiased
  estimators of $\ip{f}{\mu_s}$ and $\ip{g}{\mu_s}$, respectively, we
  have, by the Cauchy-Schwarz inequality,
  \begin{equation}     \label{eq:A}
  \begin{split}
  &\big| \E[\ip{f}{\nu_s^N}\ip{g}{\nu_s^N}] 
    - \ip{f}{\mu_s}\ip{g}{\mu_s}\big|\\
  &\qquad = \big| \E\big[(\ip{f}{\nu_s^N}-\ip{f}{\mu_s}) 
    (\ip{g}{\nu_s^N}-\ip{g}{\mu_s})\big]\big|\\
  &\qquad \leq \big(\E\big|\ip{f}{\nu_s^N}-\ip{f}{\mu_s}\big|^2\big)^{1/2}
            \big(\E\big|\ip{g}{\nu_s^N}-\ip{g}{\mu_s}\big|^2\big)^{1/2}\\
  &\qquad \leq \varepsilon_s^{N,p}\, \|f\|_{L^p(\mu_s)} \, \|g\|_{L^p(\mu_s)}
  \end{split}
  \end{equation}
  for all $0 \leq s \leq t$ and all functions $f$, $g:S \to \erre$.
  Since the last term on the right-hand side of (\ref{eq:VN}) can be
  bounded by
  \[
   \iint |H_s(y)|(q_{s,t}f(z)-q_{s,t}f(y))^2 \,\nu_s^N(dz)\,\nu_s^N(dy),
  \]
  an application of (\ref{eq:A}) yields, by (\ref{eq:centered}) and
  (\ref{eq:V}),
  \begin{align*}
  \E\big[V_{s,t}^N(f)\big] &\leq -\ip{H_s(q_{s,t}f)^2}{\mu_s}\ip{1}{\mu_s}
     -\ip{H_s}{\mu_s}\ip{q_{s,t}f^2-(q_{s,t}f)^2}{\mu_s}\\
     &\quad + \iint |H_s(y)|
     (q_{s,t}f(z)-q_{s,t}f(y))^2\,\mu_s(dz)\,\mu_s(dy)\,
     + \varepsilon_s^{N,p} R_{s,t}(f)\\
     &= V_{s,t}(f) + \varepsilon_s^{N,p} R_{s,t}(f),
  \end{align*}
  where
  \begin{align*}
  R_{s,t}(f)  &= \|H_s(q_{s,t}f)^2\|_{L^p(\mu_s)}
     + \|H_s\|_{L^p(\mu_s)} \big\|q_{s,t}f^2 - (q_{s,t}f)^2 \big\|_{L^p(\mu_s)}\\
     &\quad + \|H_s\|_{L^p(\mu_s)} \big\|(q_{s,t}f)^2\big\|_{L^p(\mu_s)}
     + 2\|H_sq_{s,t}f\|_{L^p(\mu_s)} \|q_{s,t}f\|_{L^p(\mu_s)}\\
     &\quad + \big\| H_s (q_{s,t}f)^2 \big\|_{L^p(\mu_s)}\\
     &\leq \|H_s\|_{L^p(\mu_s)} \|q_{s,t}f^2\|_{L^p(\mu_s)}
     + 6 \|H_s\|_{L^q(\mu_s)} \|q_{s,t}f\|^2_{L^{2r}(\mu_s)}.
     \qedhere
  \end{align*}
\end{proof}

In order to bound $V_{s,t}^N(f)$ uniformly over $f \in L^p(\mu_t)$
with $\|f\|_{L^p(\mu_t)} \leq 1$, one needs to be able to control
$\|q_{s,t}f\|_{L^{2r}(\mu_t)}$ in terms of $\|f\|_{L^p(\mu_t)}$.
This is possible if hypercontractivity holds and $t-s$ is
sufficiently large. Over short time intervals $[s,t]$ we apply in a
first step another rough estimate instead:

\begin{lemma}     \label{lm:short}
Let $p \geq 2$ and $N \in \mathbb{N}$. Then for $0\le s \le t$,
\[
\frac 1N\,\E\big[V_{s,t}^N(f)\big] \leq 4 \osc(H_s) \Big( 1 +
\varepsilon_s^{N,p} \exp\Big(2\int_s^t\osc(H_r)\,dr\Big)\Big)
\|f\|^2_{L^p(\mu_t)}.
\]
\end{lemma}
\begin{proof}
  Setting
  \[
  {A}_t^f\ :=\ \ip{f}{\nu_t^N}\ =\ \ip{f}{\eta_t^N}\,
  \exp \Big( -\int_0^t \ip{H_s}{\eta_s^N}\,ds \Big),
  \]
  we have $A_t^f=\ip{f}{\eta_t^N}\, {A}_t^1$ for all
  $f:S\to\erre$. Since
  \[
  \ip{f^2}{\eta_t^N} = \frac1N \sum_{i=1}^N f(X_{t,i})^2
  \leq \frac1N \Big( \sum_{i=1}^N |f(X_{t,i})| \Big)^2
  = N \ip{|f|}{\eta_t^N}^2,
  \]
  we obtain, recalling that $\eta_t^N$ is a probability measure,
  \begin{align}
  V_{s,t}^N(f)\ &\leq\ N\, ( A_s^1)^2\, \Big( (\max H_s^-+\max H_s^+)
      \ip{|q_{s,t}f|}{\eta_s^N}^2
      + \max H_s^- \ip{(q_{s,t}f^2)^{1/2}}{\eta_s^N}^2\nonumber\\
    &\quad\qquad\qquad\qquad\qquad + 2\osc(H_s)\ip{|q_{s,t}f|}{\eta_s^N}^2\Big)\nonumber\\
    &\leq\  N\osc(H_s) \Big( 3\ip{q_{s,t}|f|}{\nu_s^N}^2 +
                         \ip{(q_{s,t}f^2)^{1/2}}{\nu_s^N}^2 \Big).
                            \label{eq:feno}
  \end{align}
  Moreover, by inequality (\ref{eq:A}),
  \[
  \E \left[\ip{f}{\nu_t^N}^2\right] \leq \ip{f}{\mu_t}^2
  + \varepsilon_t^{N,p} \|f\|^2_{L^p(\mu_t)},
  \]
  hence, taking expectations on both sides of (\ref{eq:feno}), we obtain
  \begin{align*}
  \frac1N\,\E\big[V_{s,t}^N(f)\big] &\leq 3\osc(H_s)
      \big[ \ip{q_{s,t}|f|}{\mu_s}^2 + \varepsilon_s^{N,p}
      \big\|q_{s,t}|f|\big\|^2_{L^p(\mu_s)} \big]\\
  &\quad + \osc(H_s) \big[ \ip{q_{s,t}f^2}{\mu_s}
      + \varepsilon_s^{N,p} \big\|q_{s,t} f^2\big\|^2_{L^{p/2}(\mu_t)}\big]\\
  &\leq 4 \osc(H_s) \, \left[ \ip{f^2}{\mu_t}
      + \varepsilon_s^{N,p} \exp\Big(2\int_s^t\osc(H_r)\,dr\Big)\,
        \|f\|^2_{L^p(\mu_t)} \right],\\
  \end{align*}
  where we have used the fact that
  $\ip{q_{s,t}f}{\mu_s}=\ip{f}{\mu_t}$, and the estimate
  \begin{equation}\label{eq:lpbound}
  \big\| q_{s,t}f \big\|_{L^p(\mu_t)} \ \leq\  \exp\left(\int_s^t\osc(H_r)\, dr\right)\,
  \|f\|_{L^p(\mu_s)} .
  \end{equation}
  The proof of (\ref{eq:lpbound}) is elementary and can be found in
  \cite{JMAA}.
\end{proof}
Combining Proposition \ref{prop:long} and Lemma \ref{lm:short} we
obtain the following (rough) a priori estimate:
\begin{lemma}     \label{lm:rough}
  Let $p$, $q$, $r \in [2,\infty]$ be such that
  $p^{-1}=q^{-1}+r^{-1}$, and choose $\delta$ as in Theorem
  \ref{thm:main}. If
  \[
  N \geq 25 \,\max \big(1,\bar{C}_t(p,q,\delta ) \big)
  \]
  then
  \[
  \varepsilon_{t}^{N,p} < 1.
  \]
\end{lemma}
\begin{proof}
  Note that, by (\ref{eq:V}), 
  \[
  V_{s,t}(f) \leq 5\,\|H_s\|_{L^q(\mu_s)} \, \|q_{s,t}f\|_{L^{2r}(\mu_s)}
  \]
  for any $f:S \to \erre$ and $0 \leq s \leq t$. Hence Proposition
  \ref{prop:long} implies
  \[
  \E\big[V_{s,t}^N(f)\big] \leq \|H_s\|_{L^q(\mu_s)}
     C_{s,t}(p,q)^2 \|f\|^2_{L^p(\mu_t)}
     \big(5 + 7\varepsilon_s^{N,p}\big).
  \]
  Choosing $N$ as stated we get
  \[
  \frac1N \int_0^{(t-\delta)^+} \E\big[V_{s,t}^N(f)\big]\,ds
  \leq \frac{12}{25} \|f\|^2_{L^p(\mu_t)} \max\big(\varepsilon_t^{N,p},1\big).
  \]
  On the other hand, by Lemma \ref{lm:short} and since
  $17\,\delta\,\osc(H_s) \leq 1$ for any $s \leq t$, we obtain
  \begin{align*}
  \frac1N \int_{(t-\delta)^+}^t \E\big[V_{s,t}^N(f)\big]\,ds &\leq
  \frac4{17}\big( 1 + \varepsilon_t^{N,p} e^{2/17} \big)
  \|f\|^2_{L^p(\mu_t)}\\
  &< \frac12 \|f\|^2_{L^p(\mu_t)} \max\big(\varepsilon_t^{N,p},1\big).
  \end{align*}
  Hence by Proposition \ref{thm:bondo}, since $N \geq 50$, we get
  \begin{align*}
  \varepsilon_{t}^{N,p} &= \sup\Big\{ \frac1N \var_{\mu_t}(f)
     + \frac1N \int_0^t \E\big[V_{s,t}^N(f)\big]\,ds \,\Big|\;
       f:S\to\erre \text{ with } \|f\|_{L^p(\mu_r)} \leq 1,\;r\in[0,t]
  \Big\}\\
  &< \left(\frac1{50} + \frac{12}{25} + \frac12\right)
       \max\big(\varepsilon_t^{N,p},1\big).
  \qedhere
  \end{align*}
\end{proof}
The a priori estimate just obtained can be used instead of Lemma
\ref{lm:short} to estimate $\E\left[ V_{s,t}^N(f)\right] $ when
$t-s$ is small:
\begin{lemma}     \label{prop:tilde}
Let $q \in ]6,\infty]$ and $p \in ]4q/(q-2),\infty[$. Suppose that
\[
N \geq 25 \max \big(1,\bar{C}_t(\tilde{p},q,\delta)\big),
\]
where $\tilde p $ is defined by $\tilde p^{-1}=q^{-1}+(p/2)^{-1}$.
Then for $0 \leq s \leq t \leq t_0$,
\[
\E \left[ V_{s,t}^N(f)\right]\ \leq \ V_{s,t}(f) + 7\,\exp\left(
2\int_s^t\osc(H_r)\,dr\right) \|H_s\|_{L^q(\mu_s)}
\|f\|^2_{L^p(\mu_t)} . \]
\end{lemma}
\begin{proof}
  Note that $\tilde{p}^{-1}=q^{-1}+(p/2)^{-1}<1/2$ by the assumptions
  on $p$ and $q$. Applying Proposition \ref{prop:long} with $p$, $q$,
  $r$ replaced by $\tilde{p}$, $\tilde{q}:=q$, and $\tilde{r}:=p/2$,
  respectively, yields
  \[
  \E\big[V_{s,t}^N(f)\big] \leq V_{s,t}(f)
     + \big( \|H_s\|_{L^{\tilde{p}}(\mu_s)} \|q_{s,t}f^2\|_{L^{\tilde{p}}(\mu_s)}
     +6\|H_s\|_{L^q(\mu_t)} \|q_{s,t}f\|^2_{L^p(\mu_s)} \big)\,
  \varepsilon_s^{N,\tilde{p}}
  \]
  Since $\tilde{p} < \min (q,p/2)$, the claim follows by Lemma \ref{lm:rough}
  and the estimate (\ref{eq:lpbound}).
\end{proof}

We are now ready to prove the theorem:
\begin{proof}[Proof of Theorem \ref{thm:main}]
By Proposition \ref{prop:long} we have
\[
\E\big[V_{s,t}^N(f)\big] \leq V_{s,t}(f) +
7\|H_s\|_{L^q(\mu_s)} C_{s,t}(p,q)^2 \|f\|^2_{L^p(\mu_t)}
\varepsilon_t^{N,p}
\]
for any $f:S \to \erre$ and $0 \leq s \leq t$. Therefore by Proposition
\ref{thm:bondo}, Lemma \ref{prop:tilde}, and the choice of $\delta$,
\begin{align*}
&N\,\E\big| \ip{f}{\nu_t^N} - \ip{f}{\mu_t} \big|^2 =
\var_{\mu_t}(f) + \int_0^{(t-\delta )^+} \E\big[V_{s,t}^N(f)\big]\,ds
+ \int_{(t-\delta )^+}^t \E\big[V_{s,t}^N(f)\big]\,ds\\
&\qquad\leq
\var_{\mu_t}(f) + \int_0^t V_{s,t}(f)\,ds\\
&\qquad\quad + \Big[7\bar{C}_t(p,q,\delta)\varepsilon_t^{N,p}
+ 7 e^{2/17} \int_{(t-\delta)^+}^t \|H_s\|_{L^q(\mu_s)}\,ds
\Big] \, \|f\|^2_{L^p(\mu_t)}.
\end{align*}
Observing that $\|H_s\|_{L^q(\mu_s)} \leq \osc(H_s)$ and
that $7\, e^{2/17}/17<1$, we obtain (\ref{eq:main1}).\smallskip

Furthermore, by maximizing (\ref{eq:main1}) over all $f:S \to \erre$
such that $\|f\|_{L^p(\mu_t)} \leq 1$ and over $t$, we get
\[
N\varepsilon_t^{N,p} \leq 2 + v_t^p + 7\bar{C}_t(p,q,\delta)\varepsilon_t^{N,p}
\]
for all $t\in [0,t_0]$. Recalling that $N>25\bar{C}_t(p,q,\delta)$ by
assumption, we obtain
\begin{align*}
\varepsilon_t^{N,p} &\leq \frac{2+v_t^p}{N-7\bar{C}_t(p,q,\delta)}
= (2+v_t^p) \left( \frac1N + \frac{7\bar{C}_t(p,q,\delta)}{
N(N-7\bar{C}_t(p,q,\delta))} \right)\\
&\leq (2+v_t^p) \, N^{-1} \, \Big(
1 + \frac{7\cdot 25}{18} \bar{C}_t(p,q,\delta) N^{-1} \Big),
\end{align*}
which implies (\ref{eq:main2}).
\end{proof}

\section{Proofs of Theorems \ref{thm:main2} and \ref{thm:main3}}
\label{sec:p2}
\begin{proof}[Proof of Theorem \ref{thm:main2}]
  By the estimates in \cite{JMAA} we have, for $0\le s\le t\le t_0$,
  \[
  \| q_{s,t}f \|_{L^p(\mu_s)} \leq 2^{1/4} \| f \|_{L^p(\mu_s)}
  \]
  for all $f:S\to\erre$, provided
  \begin{equation}\label{eq:22}
    \lambda_s\ \ge\ \frac{p}{4}\, A_s\,+\,\frac{p(p+3)}{4}\,t_0\,
    B_s\qquad \mbox{for all }s\in [0,t_0].
  \end{equation}
  Hence, under
  this condition, we get $C_{s,t}(p) \leq 2^{1/4}$. Moreover, by
  \cite{JMAA},
  \[
  \|q_{t-\delta,t}f\|_{L^q(\mu_{t-\delta})}\ \leq\
  \exp \big( \int_{t-\delta}^t \max H_r^-\,dr \big) \,\|f\|_{L^p(\mu_t)}
  \]
  for all $f:S \to \erre$ and $0\le\delta\le t\le t_0$, provided
  \begin{equation}
    \label{eq:HYP}
    \lambda_s \geq \frac{\gamma_s}{4\delta}
    \log\frac{q-1}{p-1} \qquad \mbox{for all } s\in [0,t_0].
  \end{equation}
  Choosing $\delta = (17\omega)^{-1}$, we obtain that, for $s \leq
  t-\delta$,
  \[
  \| q_{s,t}f\|_{L^p(\mu_s)} \ =\
  \|q_{s,t-\delta}q_{t-\delta,t}f\|_{L^p(\mu_s)} \ \leq\
  2^{1/4} e^{1/17} \|f\|_{L^q(\mu_t)},
  \]
  if both (\ref{eq:22}) and (\ref{eq:HYP}) hold. Hence
  \[
  C_{s,t}(p,q) \leq 2^{1/4}e^{1/17}
  \]
  provided (\ref{eq:22}) holds and
  \[
  \lambda_s \geq \frac{\gamma_s}{4\delta} \log \max \big(
    \frac{2r-1}{p-1},\frac{2p-2}{p-2} \big)
    \qquad \mbox{for all } s \in [0,t_0].
  \]
  Since $2<\tilde p<p$ and $\tilde p^{-1}=q^{-1}+ (p/2)^{-1}$, we
  obtain similarly that $C_{s,t}(\tilde p,q) \leq 2^{1/4}e^{1/17}$
  provided (\ref{eq:22}) holds and
  \[
    \lambda_s  \geq \frac{\gamma_s}{4\delta} \log \max \big(
        \frac{p-1}{\tilde{p}-1},\frac{2\tilde{p}-2}{\tilde{p}-2}\big)
    \qquad \mbox{for all } s \in [0,t_0].
  \]
  Hence by (\ref{eq:stella}) and (\ref{eq:stst}) we obtain
  \[
  v_t(p) \leq 5\cdot 2^{1/2}\, K_t(2), \qquad
  \bar{C}_t(p,q,\delta) \leq 2^{1/2}\,e^{2/17}\,K_t(q), \qquad
  \bar{C}_t(\tilde{p},q,\delta) \leq 2^{1/2}\,e^{2/17}\,K_t(q)
  \]
  for any $t \leq t_0$. 
  The assertion now follows from Theorem \ref{thm:main}.
\end{proof}

\medskip

\begin{proof}[Proof of Lemma \ref{lm:nueta}]
  For a function $f:S\to\erre$ and $t \geq 0$ let
  $f_t:=f-\ip{f}{\mu_t}$. Then
  \[
  \ip{f_t}{\eta_t^N} = \ip{f}{\eta_t^N}-\ip{f}{\mu_t}
  \]
  and, by (\ref{eq:REWEMP}),
  \begin{equation}     \label{eq:revamp}
  \ip{f_t}{\nu_t^N} = \ip{1}{\nu_t^N} \ip{f_t}{\eta_t^N}.
  \end{equation}
  Hence
  \begin{align*}
    \E\big[ \ip{f_t}{\eta_t^N}^2 \big] &\leq 2\,\E\Big[ \big(
    \ip{f_t}{\eta_t^N}-\ip{f_t}{\nu_t^N}\big)^2\Big]
    + 2\,\E\big[ \ip{f_t}{\nu_t^N}^2 \big]\\
    &= 2\,\E\Big[ \big( \ip{1}{\nu_t^N}-1\big)^2
    \ip{f_t}{\eta_t^N}^2\Big]
    + 2\,\E\big[ \ip{f_t}{\nu_t^N}^2 \big]\\
    &\leq 2\, \| f_t\|^2_{\sup } \,\E\left[ \left(
        \ip{1}{\nu_t^N}-1\right)^2\right] + 2\,\E\left[
      \ip{f_t}{\nu_t^N}^2\right].
  \end{align*}
  Applying this bound and (\ref{eq:revamp}), we obtain the $L^1$
  estimate:
  \begin{align*}
    \E\big[ \big|\ip{f_t}{\eta_t^N}\big| \big] &=
    \E\big[ \big|\ip{f_t}{\eta_t^N} \big(1-\ip{1}{\nu_t^N}\big)\big|
      \big]
    + \E\big[\big|\ip{f_t}{\nu_t^N}\big| \big]\\
    &\leq
      \E\left[ \ip{f_t}{\eta_t^N}^2\right]^{1/2}
      \E\big[\big(
          \ip{1}{\nu_t^N}-1\big)^2\big]^{1/2} + \E\left[
        \ip{f_t}{\nu_t^N}^2\right]^{1/2}\\
    &\leq \E\left[ \ip{f_t}{\nu_t^N}^2\right]^{1/2}\, +\,\sqrt 2\,\|
    f_t\|_{\sup }\E\left[ \left( \ip{1}{\nu_t^N}-1\right)^2\right]\\
    &\quad + \sqrt 2\,\E\left[ \ip{f_t}{\nu_t^N}^2\right]^{1/2}\E\left[
      \left( \ip{1}{\nu_t^N}-1\right)^2\right]^{1/2}.
  \end{align*}
  This proves Lemma \ref{lm:nueta}.
\end{proof}

\smallskip

\begin{proof}[Proof of Corollary \ref{cor:estf}] 
  The first assertion is an immediate consequence of (\ref{eq:main1})
  and (\ref{eq:main3}). The second assertion follows by the first one
  and (\ref{eq:L1L2}).
\end{proof}

\smallskip

\begin{proof}[Proof of Theorem \ref{thm:main3}]
  Fix $i \in I$ and define
  \[
  h_t(i) := \ip{H_t}{\mu_t^i} = {\int_{S_i} H_t\,d\mu_t}/{\mu_t(S_i)}.
  \]
  Note that
  \[
  h_t(i) = - \frac{d}{dt} \log \mu_t(S_i).
  \]
  Since (\ref{eq:timedep}) and (\ref{eq:centered}) hold,
  $H_t^i=H_t-h_t(i)$ is the negative logarithmic time derivative of
  $\mu_t^i$. If we define $q_{s,t}^if$ for functions $f:S_i \to \erre$
  in the same way as $q_{s,t}f$ with $H_t$ replaced by $H_t^i$, then
  \[
  q_{s,t}f(x)\ =\ \exp \big( -\int_s^t h_r(i)\,dr \big)\, q_{s,t}^if(x)\
  =\
  \frac{\mu_t(S_i)}{\mu_s(S_i)}\, q_{s,t}^if(x).
  \]
  In particular, for $p \in [1,\infty]$, we have
  \begin{equation}     \label{eq:LOCEST}
  \| q_{s,t}f\|_{L^p(\mu_s)}^\sim \ = \ \max_{i \in I} \|q_{s,t}f\|_{L^p(\mu_s^i)}
  \ \leq\ \max_{i \in I}\, \frac{\mu_t(S_i)}{\mu_s(S_i)}\, \|q_{s,t}^if\|_{L^p(\mu_s^i)}.
  \end{equation}
  Assuming Poincar\'e and log Sobolev inequalities with respect to the
  measures $\mu_t^i$ and the functions $H_t^i$, we obtain the same
  type of $L^p$-$L^q$ bounds for the operators $q_{s,t}^i$ as we did
  for the operators $q_{s,t}$ in the proof of Theorem
  \ref{thm:main2}. Because of (\ref{eq:LOCEST}) the assertion then
  follows similarly as above.
\end{proof}

\appendix

\section{Spectral gap and LSI for 1D Metropolis}
In this appendix we prove upper bounds for the Poincar\'e and
logarithmic Sobolev constants for Random Walk Metropolis algorithms on
a finite subset $S$ of $\mathbb{Z}$. Let
$S:=\{a,a+1,\ldots,-1,0,1,\ldots,a+\Delta-1\}$ with $a \in
\mathbb{Z}$ and $\Delta \in \mathbb{N}$ such that $0 \in S$.  We
assume that $\mu$ is a probability measure on $S$
satisfying
\begin{itemize}
\item[(i)] $\mu(x) \leq \rho \mu(y)$ for any $x$, $y \in
  [-s,s]$;
\item[(ii)] $\mu(x+1) \leq \alpha\mu(x)$ for any $x \geq s$, and
$\mu(x-1) \leq \alpha\mu(x)$ for any $x \leq -s$,
\end{itemize}
for appropriate constants $s \in \mathbb{Z}_+$, $\rho \in
[1,+\infty[$, and $\alpha \in ]0,1[$.
For notational convenience, we set
\[
b := a + \Delta - 1,
\qquad
r := \frac{1}{1-\alpha} \wedge \Delta,
\qquad
u := s \wedge \Delta.
\]
The Random Walk Metropolis chain for sampling from $\mu$ is the Markov
chain on $S$ with generator $\L$ satisfying
\[
\L(x,y) =
\begin{cases}
  \frac12 \min\Big( \frac{\mu(y)}{\mu(x)},1 \Big), & \text{if } |y-x|=1,\\
  0, & \text{if } |y-x|>1.
\end{cases}
\]
To estimate the Poincar\'e constant for this dynamics, we can apply a
general upper bound for one-dimensional Markov chains due to Miclo
\cite{Miclo-H}, which implies in our case
\begin{equation}     \label{eq:app0}
C^{\mathrm{Poi}} \leq 4\,\max(B^+,B^-),
\end{equation}
where
\[
B^+ := \max_{1\leq k \leq b} B_k^+,
\qquad
B_k^+ := \sum_{x=1}^k \frac{1}{\mu(x-1) \wedge \mu(x)}
\sum_{x=k}^b \mu(x),
\]
\[
B^- := \max_{a\leq k \leq -1} B_k^-,
\qquad
B_k^- := \sum_{x=k}^{-1} \frac{1}{\mu(x+1) \wedge \mu(x)}
\sum_{x=a}^k \mu(x).
\]
The bound is sharp up to a factor $4$, see \cite{Miclo-H}.  We are
going to estimate $B_k^+$ in the cases $k>s$ and $k \leq s$
separately. Corresponding bounds hold for $B_k^-$. Let us assume first
that $k > s$. Then we have, by (ii),
\[
\sum_{x=s+1}^k \frac{1}{\mu(x-1) \wedge \mu(x)} =
\sum_{x=s+1}^k \frac{1}{\mu(x)} \leq
\frac{1}{\mu(k)} \sum_{i=0}^{k-s-1} \alpha^i
\leq \frac{r}{\mu(k)}.
\]
and, by (i) and (ii),
\[
\sum_{x=1}^s \frac{1}{\mu(x-1) \wedge \mu(x)} \leq
\frac{\rho u}{\mu(s)} \leq
\frac{\alpha^{k-s}\rho u}{\mu(k)}.
\]
Hence
\begin{equation}     \label{eq:app1}
\sum_{x=1}^k \frac{1}{\mu(x-1) \wedge \mu(x)} \leq
(r + \alpha^{k-s}\rho u) \frac1{\mu(k)}.
\end{equation}
Similarly, by (ii),
\begin{equation}     \label{eq:app12}
\sum_{x=k}^b \mu(x) \leq \mu(k) \sum_{i=0}^{b-k} \alpha^i \leq r\mu(k).
\end{equation}
Therefore (\ref{eq:app1}) and (\ref{eq:app12}) yield
\begin{equation}     \label{eq:app3}
B_k^+ \leq r \big( r + \alpha^{k-s} \rho u \big)
\leq r^2 + \rho\,u\,r
\qquad \text{for any } k > s.
\end{equation}
Let us now consider the case $k \leq s$: by (i) and since $s
\wedge b \leq u$, we have
\[
\sum_{x=1}^k \frac{1}{\mu(x-1) \wedge \mu(x)}
\sum_{x=k}^{s \wedge b-1} \mu(x)
= \sum_{x=1}^k \sum_{y=k}^{s \wedge b-1}
\frac{\mu(y)}{\mu(x-1) \wedge \mu(x)}
\leq \rho k (u-k) \leq \rho u^2/4.
\]
Moreover, similarly to (\ref{eq:app12}), we have
\[
\sum_{x=s\wedge b}^b \mu(x) \leq r \mu(s \wedge b),
\]
hence, by (i) and since $k \leq s$ and $k \leq \Delta$,
\[
\sum_{x=1}^k \frac{1}{\mu(x-1) \wedge \mu(x)}
\sum_{x=s\wedge b}^b \mu(x)
\leq r\sum_{x=1}^k \frac{\mu(s\wedge b)}{\mu(x-1) \wedge \mu(x)}
\leq \rho\,k\,r \leq \rho\,u\,r.
\]
Combining these estimates, we obtain
\begin{equation}     \label{eq:app4}
B_k^+ \leq \frac14 \rho u^2 + \rho u r,
\qquad \text{for any } k \leq s.
\end{equation}
By (\ref{eq:app3}) and (\ref{eq:app4}), we finally obtain
\[
B^+ := \max_{k=1,\ldots,b} B_k^+ \leq \rho u r + \max(r^2,\rho u^2/4).
\]
Observing that the same estimate holds for $B^-$, we have shown:
\begin{thm}     \label{thm:metro-Poi}
  The Poincar\'e constant $C^{\mathrm{Poi}}$ for the Random Walk
  Metropolis chain with stationary distribution $\mu$ satisfies
  \[
  C^{\mathrm{Poi}} \leq 4\rho u r + \max(4r^2,\rho u^2)
  \]
\end{thm}
\begin{proof}
  The result holds by the upper bound (\ref{eq:app0}).
\end{proof}

\medskip

For the corresponding logarithmic Sobolev constant the following upper
bound follows from the results in \cite{Miclo-H}:
\[
\gamma \leq 20 \max (\beta^+,\beta^-),
\]
where
\[
\beta^+ := \max_{1\leq k \leq b} \beta_k^+,
\qquad
\beta_k^+ := \sum_{x=1}^k \frac{2}{\mu(x-1) \wedge \mu(x)}
\sum_{x=k}^b \mu(x)
\Big| \log \sum_{x=k}^b \mu(x) \Big|,
\]
\[
\beta^- := \max_{a\leq k \leq -1} \beta_k^-,
\qquad
\beta_k^- := \sum_{x=k}^{-1} \frac{2}{\mu(x+1) \wedge \mu(x)}
\sum_{x=a}^k \mu(x)
\Big| \log \sum_{x=a}^k \mu(x) \Big|.
\]
Again, the bound is sharp up to an explicit numerical constant.
A rough estimate for $\beta_k^+$ can easily be obtained observing that
\[
\Big| \log \sum_{x=k}^b \mu(x) \Big| =
\log \Big( \sum_{x=k}^b \mu(x) \Big)^{-1}
\leq \log \frac1{\mu(k)} \leq \log \frac1{\mu_*},
\]
where $\mu_* = \min_x \mu(x)$. In fact, this implies
\[
\beta_k^+ \leq 2B_k^+ \log \frac1{\mu_*},
\]
hence upper bounds for $\beta^+$ and $\beta^-$ can be obtained from
the corresponding bounds for $B^+$ and $B^-$ simply by multiplying by
a factor $2 \log \mu_*^{-1}$. In particular, the upper bound for
$C^{\mathrm{Poi}}$ derived above yields an upper bound for $\gamma$:
\begin{thm}     \label{thm:metro-LS}
One has
\[
\gamma \leq 10 \big( 4\rho u r + \max(\rho u^2,4r^2)\big)
\log \frac1{\mu_*}.
\]
\end{thm}

\subsection*{Example: A discrete Gauss model}
Assume that
\[
\mu(x) \propto \exp\Big( -\frac{x^2}{2\sigma^2} \Big)
\]
for some finist constant $\sigma>0$.
Then one can check that (i) and (ii) above are satisfied with
\[
s = \lfloor \sigma \rfloor,
\qquad
\rho = e^{1/2},
\qquad
\alpha = \frac{\mu(s+1)}{\mu(s)} 
= \exp\Big(- \frac{\lfloor \sigma \rfloor + 1/2}{\sigma^2} \Big).
\]
Note that $\alpha \leq e^{-1/2}$ for $\sigma<1$ and $\alpha \leq
e^{-3/4\sigma}$ for $\sigma \geq 1$. Applying the elementary
inequality $1-e^{-x} \geq \min(2x/3,1/2)$, we obtain $1-\alpha \geq
1/(2\sigma)$ if $\sigma>1$ and $1-\alpha \geq 1/3$ if $\sigma \leq 1$. Hence
\[
r = \frac{1}{1-\alpha} \wedge \Delta \leq (2\sigma \vee 3) \wedge \Delta
\leq 2\big( (\sigma \wedge \Delta) \vee 2 \big).
\]
By Theorem \ref{thm:metro-Poi}, we then obtain
\[
C^{\mathrm{Poi}} \leq 30 \, \big( (\sigma \wedge \Delta) \vee 2\big)^2.
\]
Moreover, since $-\Delta \leq a \leq b \leq \Delta$, one has
\[
\frac{\mu(k)}{\mu(0)} = \exp\Big( -\frac{k^2}{2\sigma^2} \Big)
\geq \exp\Big( -\frac12 \frac{\Delta^2}{\sigma^2} \Big)
\qquad \text{for any } k \in S,
\]
and thus
\[
\log \frac{1}{\mu_*} \leq \frac12 \big(\Delta/\sigma\big)^2
+ \log \frac{1}{\mu(0)}
\leq \frac12 \big(\Delta/\sigma\big)^2 + \log \Delta.
\]
Therefore we obtain, by Theorem
\ref{thm:metro-LS},
\begin{align*}
\gamma &\leq  150\big( (\sigma \wedge \Delta) \vee 2\big)^2
\big(\Delta/\sigma\big)^2
+ 300 \big( (\sigma \wedge \Delta) \vee 2\big)^2 \log \Delta\\
&\leq 300 \Big( \frac{\Delta}{\sigma \wedge 1} \Big)^2
+ 300 \big( (\sigma \wedge \Delta) \vee 2\big)^2 \log \Delta.
\end{align*}

\bibliographystyle{amsplain}
\bibliography{mcmc}

\end{document}